\pdfoutput=1
\RequirePackage{ifpdf}
\ifpdf % We are running pdfTeX in pdf mode
\documentclass[pdftex]{sigma}
\else
\documentclass{sigma}
\fi

\usepackage{tensor}
\usepackage[all]{xy}

\newtheorem{Theorem}{Theorem}[section]
\newtheorem{Corollary}[Theorem]{Corollary}
\newtheorem{Lemma}[Theorem]{Lemma}
\newtheorem{Proposition}[Theorem]{Proposition}
 { \theoremstyle{definition}
\newtheorem{Definition}[Theorem]{Definition}

\newtheorem{Remark}[Theorem]{Remark} }

\numberwithin{equation}{section}

\def\mbb{\mathbb}
\def\mb{\mathbf}

\def\mr{\mathrm}
\def\mc{\mathcal}

\newlength{\equwidth}
\newcommand{\crd}{\mbox{$
\begin{picture}(9,8)(1.6,0.15)
\put(1,0.2){\mbox{$ D \hspace{-7.8pt} /$}}
\end{picture}$}}

\DeclareMathOperator{\Ric}{Ric}

\DeclareMathOperator{\tr}{tr}
\DeclareMathOperator{\id}{id}

\def\X{\mathfrak X}

\def\be{\beta}
\def\ga{\gamma}
\def\de{\delta}

\def\io{i}
\def\ka{\kappa}

\def\rh{\rho}
\def\Rho{\mbox{\textsf{P}}}
\def\si{\sigma}

\def\Ups{\Upsilon}
\def\ph{\varphi}

\def\om{\omega}
\def\Ga{\Gamma}
\def\De{\Delta}
\def\Th{\Theta}
\def\La{\Lambda}

\def\Om{\Omega}

\def\x{\times}
\def\pa{\partial}
\def\t{\otimes}

\def\del{\pa}
\def\delstar{\pa^*}

\def\goesto{\rightarrow}

\def\embed{\hookrightarrow}

\def\bdot{\bullet}
\def\na{\mathrm{\nabla}}

\def\lapl{\mathrm{\square}}

\def\im{\mathrm{im}\ }
\def\dim{\mathrm{dim\ }}
\def\G{\mathcal{G}}

\def\T{\mathcal{T}}
\def\ad{\mathrm{ad}}

\def\Ad{\mathrm{Ad}}

\DeclareMathOperator{\End}{End}
\DeclareMathOperator{\GL}{GL}
\DeclareMathOperator{\Spin}{Spin}
\DeclareMathOperator{\CSpin}{CSpin}

\def\CO{\mathrm{CO}}

\def\ti{\tilde}
\def\wt{\widetilde}

\def\rr{\ensuremath{\mathbb{R}}}

\def\HH{\ensuremath{\mathcal{H}}}

\def\one{\ensuremath{\id}}

\def\SL{\ensuremath{\mathrm{SL}}}

\def\SO{\ensuremath{\mathrm{SO}}}

\def\GL{\ensuremath{\mathrm{GL}}}

\def\so{\ensuremath{\mathfrak{so}}}

\def\g{\ensuremath{\mathfrak{g}}}

\def\y{\ensuremath{\mathfrak{y}}}
\def\p{\ensuremath{\mathfrak{p}}}
\def\q{\ensuremath{\mathfrak{q}}}
\def\gl{\ensuremath{\mathfrak{gl}}}
\def\sl{\ensuremath{\mathfrak{sl}}}

\newcommand{\bg}{\mbox{\boldmath{$ g$}}}

\DeclareMathOperator{\Hol}{Hol}

\def\mb{\mathbf}

\def\mr{\mathrm}
\def\mc{\mathcal}

\def\calG{\ensuremath{\mathcal{G}}}

\def\f{f^\circ}
\def\wtf{\wt{f}^\circ}

\def\t{\otimes}
\def\pmat#1{\begin{pmatrix}#1\end{pmatrix}}
\def\smat#1{\left(\begin{smallmatrix}#1\end{smallmatrix}\right)}

 \newcommand{\ind}{\indices}

\begin{document}

\allowdisplaybreaks

\newcommand{\arXivNumber}{1510.03337}

\renewcommand{\PaperNumber}{081}

\FirstPageHeading

\ArticleName{A Projective-to-Conformal Fef\/ferman-Type\\ Construction}

\ShortArticleName{A Projective-to-Conformal Fef\/ferman-Type Construction}

\Author{Matthias HAMMERL~$^{\dag^1}$, Katja SAGERSCHNIG~$^{\dag^2}$, Josef \v{S}ILHAN~$^{\dag^3}$,\\ Arman TAGHAVI-CHABERT~$^{\dag^4}$ and Vojt\v{e}ch \v{Z}\'ADN\'IK~$^{\dag^5}$}

\AuthorNameForHeading{M.~Hammerl, K.~Sagerschnig, J.~\v{S}ilhan, A.~Taghavi-Chabert and V.~\v{Z}\'adn\'ik}

\Address{$^{\dag^1}$~University of Vienna, Faculty of Mathematics, \\
\hphantom{$^{\dag^1}$}~Oskar-Morgenstern-Platz 1, 1010 Vienna, Austria}
\EmailDD{\href{mailto:matthias.hammerl@univie.ac.at}{matthias.hammerl@univie.ac.at}}
\Address{$^{\dag^2}$~INdAM-Politecnico di Torino, Dipartimento di Scienze Matematiche,\\
\hphantom{$^{\dag^2}$}~Corso Duca degli Abruzzi 24, 10129 Torino, Italy}
\EmailDD{\href{mailto:katja.sagerschnig@univie.ac.at}{katja.sagerschnig@univie.ac.at}}
\Address{$^{\dag^3}$~Masaryk University, Faculty of Science, Kotl\'{a}\v{r}sk\'{a} 2, 61137 Brno, Czech Republic}
\EmailDD{\href{mailto:silhan@math.muni.cz}{silhan@math.muni.cz}}
\Address{$^{\dag^4}$~Universit\`{a} di Torino, Dipartimento di Matematica ``G.~Peano'',\\
\hphantom{$^{\dag^4}$}~Via Carlo Alberto 10, 10123 Torino, Italy}
\EmailDD{\href{mailto:ataghavi@unito.it}{ataghavi@unito.it}}
\Address{$^{\dag^5}$~Masaryk University, Faculty of Education, Po\v{r}\'\i\v{c}\'\i\ 31, 60300 Brno, Czech Republic}
\EmailDD{\href{mailto:zadnik@mail.muni.cz}{zadnik@mail.muni.cz}}

\ArticleDates{Received February 09, 2017, in f\/inal form October 09, 2017; Published online October 21, 2017}

\Abstract{We study a Fef\/ferman-type construction based on the inclusion of Lie groups ${\rm SL}(n+1)$ into ${\rm Spin}(n+1,n+1)$. The construction associates a split-signature $(n,n)$-conformal spin structure to a projective structure of dimension $n$. We prove the existence of a canonical pure twistor spinor and a light-like conformal Killing f\/ield on the constructed conformal space. We obtain a complete characterisation of the constructed conformal spaces in terms of these solutions to overdetermined equations and an integrability condition on the Weyl curvature. The Fef\/ferman-type construction presented here can be understood as an alternative approach to study a conformal version of classical Patterson--Walker metrics as discussed in recent works by Dunajski--Tod and by the authors. The present work therefore gives a complete exposition of conformal Patterson--Walker metrics from the viewpoint of parabolic geometry.}


\Classification{53A20; 53A30; 53B30; 53C07}

\section{Introduction}
In \emph{conformal geometry} the geometric structure is given by an equivalence class of pseudo-Riemannian metrics: two metrics $g$ and $\hat{g}$ are considered to be equivalent if they dif\/fer by a~positive smooth rescaling, $\hat{g}=e^{2f}g$. In \emph{projective geometry} the geometric structure is given by an equivalence class of torsion-free af\/f\/ine connections: two connections $D$ and $\hat{D}$ are considered as equivalent if they share the same geodesics (as unparametrised curves). While conformal and projective structures both determine a corresponding class of af\/f\/ine connections, neither of them induces a single distinguished connection on the tangent bundle. Instead, both structures have canonically associated \emph{Cartan connections} that govern the respective geometries and encode prolonged geometric data of the respective structures. It is therefore often useful when studying projective and conformal structures to work in the framework of \emph{Cartan geometries}.

The present paper investigates a geometric construction that produces a conformal class of split-signature metrics on a $2n$-dimensional manifold arising naturally from a projective class of connections on an $n$-dimensional manifold. Split-signature conformal structures of this type have appeared in several places in the literature before. The projective-to-conformal construction studied in this paper should be understood as a generalisation of the classical \emph{Riemann extensions} of af\/f\/ine spaces by E.M.~Patterson and A.G.~Walker \cite{patterson-walker}. One of the main authors motivations for the present study was the article \cite{dunajski-tod} by M.~Dunajski and P.~Tod, where the Patterson--Walker construction was generalised to a projectively invariant setting in dimension $n=2$. On the other hand, in \cite{nurowski-sparling-cr} conformal structures of signature $(2,2)$ were constructed using Cartan connections that contain the conformal structures arising from $2$-dimensional projective structures as a special case. A~generalisation of this Cartan-geometric approach to higher dimensions can be found in \cite{nurowski-projective-metric}.

In this paper the construction is studied as an instance of a \emph{Fefferman-type construction}, as formalised in \cite{cap-constructions,cap-slovak-book}, based on an inclusion of the respective Cartan structure groups $\SL(n+1)\embed\Spin(n+1,n+1)$. We show that in the general situation $n\geq 3$ the induced conformal Cartan geometry is \emph{non-normal}. To obtain information on the conformal structure it is thus important to understand how the normal conformal Cartan connection dif\/fers from the induced one, and the main part of the paper concerns the study of this modif\/ication. We may summarise the main contributions of the paper as follows:
\begin{itemize}\itemsep=0pt
\item A comprehensive treatment of the projective-to-conformal Fef\/ferman-type construction including a discussion of the intermediate Lagrangean contact structure (Section~\ref{feff-const}) and a~comparison with Patterson--Walker metrics (Section~\ref{comparison}).
\item A thorough study of the normalisation process (Section \ref{nor&char}) and an explicit formula for the modif\/ication needed to obtain the normal conformal Cartan connection (Section~\ref{explicit}).
\item The characterisation of the conformal structures obtained via our Fef\/ferman-type construction (culminating in Theorem~\ref{thm-char-trac}).
\end{itemize}

Let us comment upon the characterisation in more detail. This is formulated in terms of a~conformal Killing f\/ield $k$ and a twistor spinor~$\chi$ on the conformal space together with a~(conformally invariant) integrability curvature condition. In Theorem~\ref{thm-char-trac} the properties of~$k$ and~$\chi$ are specif\/ied in terms of corresponding conformal tractors, which nicely ref\/lects the algebraic setup of the Fef\/ferman-type construction in geometric terms.

An alternative equivalent characterisation theorem was obtained by the authors in \cite[Theo\-rem~1]{hsstz-walker} by dif\/ferent means, namely, by direct computations based on spin calculus in the spirit of \cite{penrose-rindler-86, Taghavi-Chabert2012a}. The conformal properties are given purely in underlying terms and do not refer to tractors. In Section~\ref{alternative} (Theorem \ref{char-refin}) we indicate how this alternative characterisation can be obtained in the current framework.

We remark that, to our knowledge, the present work is the f\/irst comprehensive treatment of a non-normal Fef\/ferman-type construction and we expect that the techniques developed should have considerable scope for applications to other similar constructions. A particularly interes\-ting case of this sort is the Fef\/ferman construction for (non-integrable) almost CR-structures. Possible further applications concern relations between solutions of so-called \emph{BGG-equations} and special properties of the induced conformal structures. Several such relationships were already obtained by the authors in~\cite{hsstz-walker}. For instance, we can give a full description of Einstein metrics contained in the resulting conformal class in terms of the initial projective structure. Moreover, in \cite{hsstz-ambient} we were able to show that the obstruction tensor of the induced conformal structure vanishes.

\section{Projective and conformal parabolic geometries}\label{sec-background}

The standard reference for the background material on Cartan and parabolic geometries presented here is \cite{cap-slovak-book}.

\subsection{Cartan and parabolic geometries}\label{pargeo}
Let $G$ be a Lie group with Lie algebra $\g$ and $P\subseteq G$ a closed subgroup with Lie algebra $\p$. A \emph{Cartan geometry} $(\mathcal{G},\omega)$ of type $(G,P)$ over a smooth manifold $M$ consists of a $P$-principal bundle $\mathcal{G}\to M$ together with a \emph{Cartan connection} $\omega\in\Omega^1(\mathcal{G},\g)$. The canonical principal bundle $G\to G/P$ endowed with the Maurer--Cartan form constitutes the \emph{homogeneous model} for Cartan geometries of type $(G,P)$.

The \emph{curvature} of a Cartan connection $\omega$ is the $2$-form
\begin{gather*}
K\in\Omega^2(\mathcal{G},\g), \qquad K(\xi,\eta):=d\omega(\xi,\eta)+[\omega(\xi),\omega(\eta)],
\qquad \text{for all $\xi,\eta\in\mathfrak{X}(\mathcal{G})$},
\end{gather*}
which is equivalently encoded in the $P$-equivariant \emph{curvature function}
\begin{gather} \label{curvatureK}
\kappa\colon \ \mathcal{G}\to\Lambda^2(\g/\p)^*\otimes \g, \qquad \kappa(u)(X+\p,Y+\p):=K\big(\omega^{-1}(u)(X),\omega^{-1}(u)(Y)\big).
\end{gather}
The curvature is a complete obstruction to a local equivalence with the homogeneous model. If the image of $\kappa$ is contained in $\Lambda^2(\g/\p)^*\otimes \p$ the Cartan geometry is called \emph{torsion-free}.

A \emph{parabolic geometry} is a Cartan geometry of type $(G,P)$, where~$G$ is a~semi-simple Lie group and $P\subseteq G$ is a parabolic subgroup. A~subalgebra $\p\subseteq \g$ is parabolic if and only if its maximal nilpotent ideal, called nilradical $\p_{+}$, coincides with the orthogonal complement $\p^{\perp}$ of $\p\subseteq\g$ with respect to the Killing form. In particular, this yields an isomorphism $(\g/\p)^*\cong\p_+$ of $P$-modules. The quotient $\g_0=\p/\p_+$ is called the Levi factor; it is reductive and decomposes into a semi-simple part $\g_0^{ss}=[\g_0,\g_0]$ and the center $\mathfrak{z}(\g_0)$. The respective Lie groups are $G_0^{ss}\subseteq G_0\subseteq P$ and $P_+\subseteq P$ so that $P=G_0\ltimes P_+$ and $P_{+}=\exp(\p_{+})$. An identif\/ication of $\g_0$ with a subalgebra in $\p$ yields a grading $\g=\g_{-k}\oplus\cdots \oplus\g_{-1}\oplus \g_0\oplus \g_{1}\oplus\cdots\oplus\g_{k}$, where $\p_+=\g_{1}\oplus\cdots\oplus\g_{k}$. We set $\g_{-}=\g_{-k}\oplus\cdots \oplus\g_{-1}$. If $k$ is the depth of the grading the parabolic geometry is called \emph{$|k|$-graded}.

The grading of $\g$ induces a grading on $\La^2\p_+\otimes\g\cong\La^2(\g/\p)^*\otimes\g$. A parabolic geometry is called \emph{regular} if the curvature function $\ka$ takes values only in the components of positive homogeneity. In particular, any torsion-free or $|1|$-graded parabolic geometry is regular.

Given a $\mathfrak{g}$-module $V$, there is a natural $\p$-equivariant map, the \emph{Kostant co-differential},
\begin{gather}\label{kostant}
 \partial^*\colon \ \Lambda^k(\g/\p)^*\otimes V \to\Lambda^{k-1}(\g/\p)^*\otimes V,
\end{gather}
def\/ining the Lie algebra homology of $\mathfrak{p}_+$ with values in $V$; see, e.g., \cite[Section~3.3.1]{cap-slovak-book} for the explicit form.
For $V=\g$, this gives rise to a natural normalisation condition: parabolic geometries satisfying $\partial^*(\kappa)=0$ are called \emph{normal}.
The \emph{harmonic curvature} $\kappa_{H}$ of a normal parabolic geometry is the image of $\kappa$ under the projection $\ker \partial^*\to\ker \partial^*/\operatorname{im} \partial^*$. For regular and normal parabolic geometries, the entire curvature $\kappa$ is completely determined just by $\kappa_{H}$.

A \emph{Weyl structure} $j\colon \G_0{\embed}\G$ of a parabolic geometry $(\G,\om)$ over $M$ is a reduction of the $P$-principal bundle $\G\to M$ to the Levi subgroup $G_0\subseteq P$. The class of all Weyl structures, which are parametrised by one-forms on~$M$, includes a particularly important subclass of \emph{exact Weyl structures}, which are parametrised by functions on~$M$: For $|1|$-graded parabolic geometries, these correspond to further reductions of $\G_0\to M$ just to the semi-simple part~$G_0^{ss}$ of~$G_0$ or, equivalently, to sections of the principal $\rr_+$-bundle $\G_0/G_0^{ss}\to M$. The latter bundle is called the \emph{bundle of scales} and its sections are the \emph{scales}.

For a Weyl structure $j\colon \G_0\embed\G$, the pullback $j^*\om= j^*\om_{-}+j^*\om_{0}+j^*\om_{+}$ of the Cartan connection may be decomposed according to $\g=\g_-\oplus\g_0\oplus\p_+$. The $\g_0$-part $j^*\om_{0}$ is a principal connection on the $G_0$-bundle $\G_0\to M$; it induces connections on all associated bundles, which are called \emph{$($exact$)$ Weyl connections}. The $\p_+$-part $j^*\om_{+}$ is the so-called \emph{Schouten tensor}.

\subsection{Tractor bundles and BGG operators}
Every Cartan connection $\omega$ on $\G\to M$ naturally extends to a principal connection $\hat{\omega}$ on the $G$-principal bundle $\hat{\mathcal{G}}:=\mathcal{G}\times_{P}G\to M$, which further induces a linear connection $\nabla^{\mathcal{V}}$ on any associated vector bundle
$\mathcal{V}:=\mathcal{G}\times_{P} V=\hat{\mathcal{G}}\times_{G} V$ for a $G$-representation $V$. Bundles and connections arising in this way are called \emph{tractor bundles} and \emph{tractor connections}. The tractor connections induced by normal Cartan connections are called normal tractor connections.

In particular, for the adjoint representation we obtain the \emph{adjoint tractor bundle} $\mathcal{A}M := \mathcal{G}\times_{P} \mathfrak{g}$. The canonical projection $\g\to\g/\p$ and the identif\/ication $TM\cong\G\x_P(\g/\p)$ yield a~bundle projection $\Pi\colon \mc{A}M\to TM$; the inclusion $\p_+\subseteq \g$ and the identif\/ication $\p_+\cong(\g/\p)^*$ yield a bundle inclusion $T^*M\embed\mc{A}M$. This allows us to interpret the Cartan curvature $\kappa$ from~\eqref{curvatureK} as a 2-form $\Om$ on $M$ with values in~$\mc{A}M$.

The holonomy group of the principal connection $\hat\om$ is by def\/inition the \emph{holonomy of the Cartan connection}~$\om$, i.e., $\Hol(\om):=\Hol(\hat\om)\subseteq G$. By the holonomy of a geometric structure we mean the holonomy of the corresponding normal Cartan connection.

In \cite{BGG-2001}, and later in a simplif\/ied manner in \cite{BGG-Calderbank-Diemer}, it was shown that for a tractor bundle $\mc{V}=\G\x_P V$ one can associate a sequence of dif\/ferential operators, which are intrinsic to the given parabolic geometry~$(\G,\om)$,
\begin{gather*}
 \Ga(\HH_0)\overset{{\Th_0^{\mc{V}}}}{\goesto}\Ga(\HH_1)\overset{{\Th_1^{\mc{V}}}}{\goesto}\cdots\overset{{\Th_{n-1}^{\mc{V}}}}{\goesto} \Ga(\HH_n).
\end{gather*}
The operators ${\Th_k^{\mc{V}}}$ are the \emph{BGG-operators} and they operate between the sections of subquotients $\HH_k=\ker\pa^*/\im\pa^*$ of the bundles of $\mc{V}$-valued $k$-forms, where $\pa^*\colon \La^k T^*M\otimes\mc{V}\to\La^{k-1} T^*M\otimes\mc{V}$ denotes the bundle map induced by the Kostant co-dif\/ferential \eqref{kostant}.

The f\/irst BGG-operator ${\Th_0^{\mc{V}}}\colon \Ga(\HH_0)\goesto\Ga(\HH_1)$ is constructed as follows. The bundle $\HH_0$ is simply the quotient $\mc{V}/\mc{V}'$, where $\mc{V}'\subseteq\mc{V}$ is the subbundle corresponding to the largest $P$-invariant f\/iltration component in the $G$-representation~$V$. It turns out, there is a distinguished dif\/ferential operator that splits the projection $\Pi_0\colon \mc{V}\goesto\HH_0$, namely, the \emph{splitting operator}, which is the unique map $L_0^{\mc{V}}\colon \Ga(\mc{H}_0) \goesto \Ga(\mc{V})$ satisfying
\begin{gather*}
\Pi_0(L_0^{\mc{V}}(\si))=\si,\qquad \pa^* (d^{\na^{\mc{V}}} L_0^{\mc{V}}(\si)) =0, \qquad \text{for all $\si \in \Ga(\mc{H}_0)$}.
\end{gather*}
The latter condition allows to def\/ine the f\/irst BGG-operator by ${\Th_0^{\mc{V}}} := \Pi_1 \circ d^{\na^{\mc{V}}} \circ L_0^{\mc{V}}$, where $\Pi_1\colon \ker\pa^* \to \Ga(\mc{H}_1)$. The f\/irst BGG-operator def\/ines an overdetermined system of dif\/ferential equations on $\si\in\Ga(\HH_0)$, ${\Th_0^{\mc{V}}}(\si)=0$, which is termed the \emph{first BGG-equation}.

\subsection{Further notations and conventions}
In order to distinguish various objects related to projective and conformal structures, the symbols referring to conformal data will always be endowed with tildes. To write down explicit formulae, we employ abstract index notation, cf., e.g., \cite{penrose-rindler-84}. Furthermore, we will use dif\/ferent types of indices for projective and conformal manifolds. E.g., on a projective manifold $M$ we write $\mbb{E}_A:=T^*M$, $\mbb{E}^A:=TM$, and multiple indices denote tensor products, as in $\mbb{E}\ind{_A^B}:=T^*M\t TM$. Indices between squared brackets are skew, as in $\mbb{E}_{[AB]}:=\La^2 T^*M$, and indices between round brackets are symmetric, as in $\mbb{E}^{(AB)}:=S^2 TM$. Analogously, on a conformal manifold $\wt{M}$ we write $\wt{\mbb{E}}_a:=T^*\wt{M}$, $\wt{\mbb{E}}^a:=T\wt{M}$ etc. By $\mbb{E}(w)$ and $\wt{\mbb{E}}[w]$ we denote the density bundle over~$M$ and~$\wt{M}$, respectively. Tensor products with other natural bundles are denoted as $\mbb{E}_A(w):=\mbb{E}_A\t\mbb{E}(w)$, $\wt{\mbb{E}}_{[ab]}[w]:=\wt{\mbb{E}}_{[ab]}\t\wt{\mbb{E}}[w]$, and the like.

\subsection{Projective structures}\label{section-projective}
Let $M$ be a smooth manifold of dimension $n\geq 2$. A \emph{projective structure} on $M$ is given by a~class, $\mb{p}$, of torsion-free projectively equivalent af\/f\/ine connections: two connections~$D$ and~$\hat D$ are projectively equivalent if they have the same geodesics as unparametrised curves. This is the case if and only if there is a one-form $\Ups_A\in\Ga(\mbb{E}_A)$ such that, for all $\xi^A\in\Ga\big(\mbb{E}^A\big)$,
\begin{gather*}
 \hat D_A\xi^B=D_A\xi^B+\Ups_A\xi^B+\Ups_P\xi^P \de\ind{_A^B}.
\end{gather*}

An \emph{oriented projective structure} $(M,\mb{p})$, which is a projective structure $\mb{p}$ on an oriented manifold $M$, is equivalently encoded as a normal parabolic geometry of type~$(G,P)$, where $G=\SL(n+1)$ and $P=\GL_+(n)\ltimes{\rr^n}^*$ is the stabiliser of a ray in the standard representa\-tion~$\mathbb{R}^{n+1}$.

Af\/f\/ine connections from the projective class $\mb{p}$ are precisely the Weyl connections of the corresponding parabolic geometry.
Exact Weyl connections are those $D\in\mb{p}$ which preserve a~volume form~--- these are also known as \emph{special} af\/f\/ine connections.
In particular, a choice of $D\in\mb{p}$ reduces the structure group to $G_0=\GL_+(n)$, if $D$ is special, the structure group is further reduced to $G_0^{ss}=\SL(n)$.

For later purposes we now give explicit expressions of the main curvature quantities, cf., e.g., \cite{thomass, eastwood-notes}. For $D\in\mb{p}$, the Schouten tensor is determined by the Ricci curvature of $D$; if $D$ is special, then the Schouten tensor is $\Rho_{AB}=\frac{1}{n-1}R\ind{_{PA}^P_B}$, in particular, it is symmetric. The projective Weyl curvature and the Cotton tensor are
\begin{gather*}	
 W\ind{_{AB}^C_D} =R\ind{_{AB}^C_D} +\Rho_{AD}\de\ind{^C_{B}}- \Rho_{BD}\de\ind{^C_{A}}, \qquad Y_{CAB}=2D_{[A}\Rho_{B]C}.
\end{gather*}

Henceforth, we use a suitable normalisation of densities so that the line bundle associated to the canonical one-dimensional representation of~$P$ has projective weight $-1$. Hence, comparing with the usual notation, the \emph{density bundle of projective weight $w$}, denoted by $\mbb{E}(w)$, is just the bundle of ordinary $\big(\frac{-w}{n+1}\big)$-densities. As an associated bundle to $\G\to M$, $\mbb{E}(w)$ corresponds to the $1$-dimensional representation of $P$ given by
\begin{gather}\label{projectivedens}
 \GL_+(n)\ltimes{\rr^n}^* \goesto \rr_+, \qquad (A,X)\mapsto \det(A)^{w}.
\end{gather}

The \emph{projective standard tractor bundle} is the tractor bundle associated to the standard representation of $G=\SL(n+1)$. The \emph{projective dual standard tractor bundle} is denoted by $\mc{T}^*$, i.e., $\mc{T}^*:=\G\times_P{\rr^{n+1}}^*$. With respect to a choice of $D\in\mb{p}$, we write
\begin{gather*}
\mc{T}^*=
\begin{pmatrix}
\mbb{E}_A(1) \\
\mbb{E}(1)
\end{pmatrix},\qquad
\na^{\mc{T}^*}_C
\begin{pmatrix}
\ph_A \\
\si
\end{pmatrix}
=
\begin{pmatrix}
D_C\ph_A +\Rho_{CA}\si \\
D_C\si-\ph_C
\end{pmatrix}.
\end{gather*}

\subsection{Conformal spin structures and tractor formulas}\label{section-conformalspin}
Let $\wt{M}$ be a smooth manifold of dimension $2n\ge 4$. A \emph{conformal structure} of signature $(n,n)$ on~$\wt{M}$ is given by a class, $\mb{c}$, of conformally equivalent pseudo-Riemannian metrics of signature~$(n,n)$: two metrics~$g$ and~$\hat{g}$ are conformally equivalent if
$\hat{g}=f^2 g$ for a nowhere-vanishing smooth function~$f$ on~$\wt{M}$. It may be equivalently described as a reduction of the frame bundle of~$\wt{M}$ to the structure group $\CO(n,n)=\rr_+\times \SO(n,n)$. An \emph{oriented conformal structure} of signature $(n,n)$ is a conformal structure of signature $(n,n)$ together with f\/ixed orientations both in time-like and space-like directions, equivalently, a reduction of the frame bundle to the group $\CO_{\rm o}(n,n)=\rr_+\times \SO_{\rm o}(n,n)$, the connected component of the identity. An equivariant lift of such a reduction with respect to the 2-fold covering $\CSpin(n,n) =\rr_+\times \Spin(n,n) \rightarrow \CO_{\rm o}(n,n)$ is referred to as a \emph{conformal spin structure} $\big(\wt{M},\mb{c}\big)$ of signature $(n,n)$.

A conformal spin structure of signature $(n,n)$ is equivalently encoded as a normal parabolic geometry of type $\big(\wt{G},\wt{P}\big)$, where $\wt{G}=\Spin(n+1,n+1)$ and $\wt{P}=\CSpin(n,n)\ltimes{\rr^{n,n}}^*$ is the stabiliser of an isotropic ray in the standard representation~$\mathbb{R}^{n+1,n+1}$.

A general Weyl connection is a torsion-free af\/f\/ine connection $\wt{D}$ such that $\wt{D} g\in\mb{c}$ for any $g\in\mb{c}$. If $\wt{D}g=0$, i.e., $\wt{D}$ is the Levi-Civita connection of a metric $g\in\mb{c}$, it is an exact Weyl connection. A choice of Weyl connection reduces the structure group to $\wt{G}_0=\CSpin(n,n)$. If the Weyl connection is exact the structure group is further reduced to $\wt{G}^{ss}_0=\Spin(n,n)$.

Now we brief\/ly introduce the main curvature quantities of conformal structures, cf., e.g.,~\cite{eastwood-notes-conformal}. For $g\in \mb{c}$, the \emph{Schouten\ tensor},
\begin{gather*}
 \wt{\Rho}=\wt{\Rho}(g)=\frac{1}{2n-2}\left(\wt{\mr{Ric}}(g)-\frac{\wt{\mr{Sc}}(g)}{2(2n-1)}g\right),
\end{gather*}
is a trace modif\/ication of the Ricci curvature $\wt{\mr{Ric}}(g)$ by a multiple of the scalar curvature $\wt{\mr{Sc}}(g)$; its trace is denoted $\wt{J}=g^{pq}\wt{\Rho}_{pq}$. The conformal Weyl curvature and the Cotton tensors are
\begin{gather*}
 \wt{W}_{ab\; d}^{\;\;\; c}=\wt{R}_{ab\; d}^{\;\;\; c}-2\de_{[a}^c\wt{\Rho}_{b]d}+2g_{d[a}\wt{\Rho}_{b]}^{\; c}, \qquad
 \wt{Y}_{cab} = 2\wt{D}_{[a} \wt{\Rho}_{b]c}.
\end{gather*}

As for projective structures, we will employ a suitable parametrisation of densities so that the canonical 1-dimensional representation of $\wt{P}$ has conformal weight $-1$. Hence, the \emph{density bundle of conformal weight $w$}, denoted as $\wt{\mbb{E}}[w]$, is just the bundle of ordinary $\big(\frac{-w}{2n}\big)$-densities. As an associated bundle to the Cartan bundle $\wt{\G}\to\wt{M}$, it corresponds to the $1$-dimensional representation of $\wt{P}$ given by
\begin{gather}\label{conformaldens}
 (\rr_+\times \Spin(n,n))\ltimes{\rr^{2n}}^* \to \rr_+, \qquad (a,A,Z)\mapsto a^{-w}.
\end{gather}
In particular, the conformal structure may be seen as a section of $\wt{\mbb{E}}_{(ab)}[2]$, which is called the \emph{conformal metric} and denoted by $\bg_{ab}$.

The spin bundles corresponding to the irreducible spin representations of $\Spin(n,n)$ are denoted by $\wt{\Sigma}_+$ and $\wt{\Sigma}_-$, and $\wt{\Sigma}=\wt{\Sigma}_+\oplus\wt{\Sigma}_-$. We employ the weighted conformal gamma matrix $\ga\in \Ga\big(\wt{\mbb{E}}_a\t\big(\End\wt{\Sigma}\big)[1]\big)$ such that $\ga_p\ga_q+\ga_q\ga_p=-2\bg_{pq}$. For $\xi\in\X\big(\wt{M}\big)$ and $\chi\in\Ga\big(\wt{\Sigma}\big)$, the Clif\/ford multiplication of $\xi$ on $\chi$ is then written as $\xi\cdot\chi=\xi^p \ga_p\chi$.

The \emph{conformal standard tractor bundle} is the associated bundle $\wt{\mc{T}}:=\wt\G\times_{\wt P}\rr^{n+1,n+1}$ with respect to the standard representation. It carries the canonical tractor metric $\mb{h}$ and the conformal standard tractor connection $\wt{\na}^{\wt{\mc{T}}}$, which preserves $\mb{h}$. With respect to a metric $g\in\mb{c}$, we have
\begin{gather} \label{constd}
 \wt{\mc{T}}=\begin{pmatrix}
 \wt{\mbb{E}}[-1]\\
 \wt{\mbb{E}}_a[1] \\
 \wt{\mbb{E}}[1]
\end{pmatrix}, \quad \mb{h}=
 \begin{pmatrix}
 0 & 0 & 1 \\
 0 & \bg & 0 \\
 1 & 0 & 0
\end{pmatrix}, \qquad \wt{\na}^{\wt{\mc{T}}}_c
 \begin{pmatrix}
 \rh \\
 \ph_a \\
 \si
 \end{pmatrix}
 =
 \begin{pmatrix}
 \wt{D}_c \rh-\wt{\Rho}_{c}^{\; b}\ph_b \\
 \wt{D}_c\ph_a+\si \wt{\Rho}_{ca}+\rh \bg_{c a}\\
 \wt{D}_c \si-\ph_c
 \end{pmatrix}.
\end{gather}
The BGG-splitting operator is given by
\begin{gather}\label{splitStd} L_0^{\wt{\mc{T}}}\colon \ \Ga\big(\wt{\mbb{E}}[1]\big)\goesto \Ga\big(\wt{\mc{T}}\big), \qquad
 \si\mapsto
 \begin{pmatrix}
 \frac{1}{2n}\big({-}\wt{D}^p\wt{D}_p-\wt{J}\big)\si \\
 \wt{D}_a\si \\
 \si
 \end{pmatrix}.
\end{gather}

The \emph{spin tractor bundle} is the associated bundle $\wt{\mc{S}}:=\wt\G\times_{\wt P}\De^{n+1,n+1}$, where $\De^{n+1,n+1}$ is the spin representation of $\wt{G}=\Spin(n+1,n+1)$. Since we work in even signature, it decomposes into irreducibles $\De^{n+1,n+1}=\De^{n+1,n+1}_+\oplus\De^{n+1,n+1}_-$; the corresponding bundles are denoted by $\wt{\mc{S}}_{\pm}=\wt\G\times_{\wt P}\De_{\pm}^{n+1,n+1}$. Under a choice of $g\in\mb{c}$, these decompose as $ \wt{\mc{S}}_{\pm}= \left(\begin{smallmatrix}
 \wt{\Sigma}_{\mp}[-\frac{1}{2}] \\
 \wt{\Sigma}_{\pm}[\frac{1}{2}]
 \end{smallmatrix}\right)$, where~$\wt{\Sigma}_{\pm}$ are the natural spin bundles as before. For later use we record the formulas for the Clif\/ford action of $\wt{\mc{T}}$ on $\wt{\mc{S}}$ and for the spin tractor connections on $\wt{\mc{S}}=\wt{\mc{S}}_{+}\oplus\wt{\mc{S}}_-$,
\begin{gather}\label{traCli}
 \begin{pmatrix} \rh \\ \ph_a \\ \si
 \end{pmatrix} \cdot
 \begin{pmatrix} \tau \\ \chi
 \end{pmatrix}
 =
 \begin{pmatrix} -\ph_a \ga^a \tau+\sqrt{2}\rh \chi \\
\ph_a \ga^a \chi-\sqrt{2}\si \tau
 \end{pmatrix}, \qquad
 \wt{\na}^{\wt{\mc{S}}}_c
 \begin{pmatrix}
 \tau \\
 \chi
 \end{pmatrix}
 =
 \begin{pmatrix}
 \wt{D}_c\tau+\frac{1}{\sqrt{2}}\wt{\Rho}_{cp}\ga^p\chi\\
 \wt{D}_c\chi+\frac{1}{\sqrt{2}}\ga_c\tau
 \end{pmatrix},
\end{gather}
cf.~\cite{mrh-coupling}. The BGG-splitting operator of $\wt{\mc{S}}_{\pm}$ is
\begin{gather}\label{twisplit}
 L_0^{\wt{\mc{S}}_{\pm}}\colon \ \Ga\big(\wt{\Sigma}_{\pm}\big[\tfrac{1}{2}\big]\big) \goesto \Ga\big(\wt{\mc{S}}_{\pm}\big), \qquad
 \chi\mapsto
 \begin{pmatrix}
 \tfrac{1}{\sqrt{2}n}\crd\chi \\
 \chi
 \end{pmatrix},
\end{gather}
where
$\crd\colon \Ga\big(\wt{\Sigma}_{\pm}\big)\goesto \Ga\big(\wt{\Sigma}_{\mp}\big)$, $\crd:=\ga^p\wt{D}_p$, is the \emph{Dirac} operator. The f\/irst BGG-operator associated to $\wt{\mc{S}}_{\pm}$ is the \emph{twistor operator}
\begin{gather*}
\Th^{\wt{\mc{S}}}_0 \colon \ \Ga\big(\wt{\Sigma}_{\pm}\big[\tfrac{1}{2}\big]\big)\goesto \Ga\big(\wt{\mbb{E}}_a\t \wt{\Sigma}_{\pm}\big[\tfrac{1}{2}\big]\big),\qquad \chi \mapsto\wt{D}_a\chi+\tfrac{1}{2n}\gamma_a\crd\chi,
\end{gather*}
cf., e.g.,~\cite{baum-friedrich-twistors}. Elements in the kernel of $\Th^{\wt{\mc{S}}}_0$ are called \emph{twistor spinors}. It is well known that ${\Pi}^{\wt{\mc{S}}}_0$ induces an isomorphism between $\wt{\na}^{\wt{\mc{S}}}$-parallel sections of $\wt{\mc{S}}$ with $\ker \Th^{\wt{\mc{S}}}_0$.

The \emph{adjoint tractor bundle} is the associated bundle $\mc{A}\wt{M}:=\wt{\mc{\G}}\times_{\wt P}\ti\g$ with respect to the adjoint representation of $\wt{G}$ on $\ti\g=\so(n+1,n+1)\cong\La^2 \rr^{n+1,n+1}$. The standard pairing on $\mc{A}\wt{M}$ induced by the Killing form on $\ti\g$ is denoted as $\langle \cdot,\cdot\rangle\colon \mc{A}\wt{M}\x\mc{A}\wt{M}\to\rr$. Henceforth we identify~$\mc{A}\wt{M}$ with~$\La^2\wt{\mc{T}}$. With respect to a~metric~$g\in \mb{c}$,
\begin{gather*}
\mc{A}\wt{M}=
 \begin{pmatrix}
 \wt{\mbb{E}}_{a}[0] \\
 \wt{\mbb{E}}_{[a_0a_{1}]}[2]\; \big|\; \wt{\mbb{E}}[1] \vspace{1mm}\\
 \wt{\mbb{E}}_{a}[2]
 \end{pmatrix}.
\end{gather*}
The standard representation of $\wt{\mathfrak{g}}$ on $\mbb{R}^{n+1,n+1}$ gives rise to the map
\begin{gather}\label{adjact}
\bullet\colon \ \mc{A}\wt{M} \otimes \wt{\mc{T}} \to \wt{\mc{T}}, \qquad
\begin{pmatrix} \rh_{a} \\ \mu_{a_0 a_1}\; \big |\; \ph \\ \be_{a} \end{pmatrix} \bullet
\begin{pmatrix} \nu \\ \om_b \\ \si \end{pmatrix} =
\begin{pmatrix} \rh^r\om_r - \ph\nu\\ \mu_b{}^r\om_r - \si\rh_b - \nu\be_b \\ \be^r\om_r + \ph\si \end{pmatrix}.
\end{gather}
The normal tractor connection is given by
\begin{gather} \label{conadj}
 \wt{\na}^{\mc{A}\wt{M}}_c
 \begin{pmatrix}
 \rh_{a} \\
 \mu_{a_0 a_1}\; \big|\; \ph \\
 k_{a}
 \end{pmatrix}=
\begin{pmatrix}
 \wt{D}_c\rh_{a} -\wt{\Rho}_{c}^{\; p}\mu_{pa}-\wt{\Rho}_{ca}\ph \\
 \begin{pmatrix}\wt{D}_c\mu_{a_0 a_1} +2\bg_{c[a_0}\rh_{a_1]} \\
 +2\wt{\Rho}_{c[a_0} k_{a_1]}\end{pmatrix}
 \; \big| \; \big(
 \wt{D}_c\ph -\wt{\Rho}_{c}^{\; p} k_{p}
 +\rh_{c}\big)
 \\
 \wt{D}_c k_{a}
 - \mu_{ca}+\bg_{ca}\ph
 \end{pmatrix}.
\end{gather}
Written as a two-form $\wt{\Om}$ with values in $\La^2\wt{\mc{T}}$, the curvature of~$\wt{\na}^{\wt{\mc{T}}}$ is
\begin{gather} \label{curvature}
\wt{\Om}_{c_0c_1}=
\begin{pmatrix}
-\wt{Y}_{ac_0c_1} \\ \wt{W}_{c_0c_1a_0a_1}\;\big|\; 0 \\ 0
\end{pmatrix}
\in \Ga\big(\wt{\mbb{E}}_{[c_0c_1]} \otimes \mc{A}\wt{M}\big).
\end{gather}
The BGG-splitting operator
\begin{gather*}
 L_0^{\mc{A}\wt{M}}\colon \ \Ga\big(\wt{\mbb{E}}^a\big)=\Ga\big(\wt{\mbb{E}}_a[2]\big)\goesto \Ga\big(\mc{A}\wt{M}\big),\qquad
 k_a\mapsto
 \begin{pmatrix}
 \rh_{a} \\
 \mu_{a_0 a_1} \; \big|\; \ph \\
 k_{a}
 \end{pmatrix},
\end{gather*}
is determined by
 \begin{gather}\label{L0La2}
 \mu_{a_0 a_1}=\wt{D}_{[a_0} k_{a_1]},\qquad \ph=-\frac{1}{2n}\bg^{pq}\wt{D}_p k_{q},\\ \notag
 \rho_a= -\frac{1}{4n}\wt{D}^p\wt{D}_p k_{a} +\frac{1}{4n}\wt{D}^p\wt{D}_{a} k_{p} +\frac{1}{4n^2}\wt{D}_{a}\wt{D}^p k_{p} +\frac{1}{n}\wt{\Rho}^p_{\; a} k_{p} -\frac{1}{2n}\wt{J} k_{a},
\end{gather}
and the corresponding f\/irst BGG-operator of $\mc{A}\wt{M}$ is computed as
\begin{gather*}
 \Th^{\mc{A}\wt{M}}_0\colon \ \Ga\big(\wt{\mbb{E}}_a[2]\big)\goesto \Ga\big(\wt{\mbb{E}}_{(ab)_0}[2]\big),\qquad \xi_a\mapsto\wt{D}_{(c}\xi_{a)_0},
\end{gather*}
where the subscript $0$ denotes the trace-free part. Thus $\Th^{\mc{A}\wt{M}}_0$ is the conformal Killing operator and solutions to the f\/irst BGG-equation are conformal Killing f\/ields. In a prolonged form, the conformal Killing equation is equivalent to
\begin{gather}\label{cKf}
 \wt{\na}^{\mc{A}\wt{M}}_b s= \xi^a\wt{\Om}_{ab},
\end{gather}
where $s= L_0^{\mc{A}\wt{M}}(\xi)$, see \cite{cap-infinitaut, gover-2006}.

\section{The Fef\/ferman-type construction}\label{feff-const}
The construction of split-signature conformal structures from projective structures discussed in this section f\/its into a general scheme relating parabolic geometries of dif\/ferent types. Namely, it is an instance of the so-called Fef\/ferman-type construction, whose name and general procedure is motivated by Fef\/ferman's construction of a canonical conformal structure induced by a CR structure, see~\cite{cap-constructions} and~\cite{cap-slovak-book} for a detailed discussion.

\subsection{General procedure}\label{setup}
Suppose we have two pairs of semi-simple Lie groups and parabolic subgroups, $(G,P)$ and~$\big(\wt{G},\wt{P}\big)$, and a Lie group homomorphism $i\colon G\to\wt{G}$ such that the derivative $i'\colon \g\to\tilde{\g}$ is injective. Assume further that the $G$-orbit of the origin in~$\wt{G}/\wt{P}$ is open and that the parabolic $P\subseteq G$ contains $Q:=i^{-1}\big(\wt{P}\big)$, the preimage of~$\wt{P}\subseteq\wt{G}$.

Given a parabolic geometry $(\mathcal{G}\to M,\omega)$ of type $(G,P)$, one f\/irst forms the \emph{Fefferman space}
\begin{gather}\label{Mtilde}
 \wt{M}:=\mathcal{G}/Q=\mathcal{G}\times_{P}P/Q.
\end{gather}
Then $\big(\mathcal{G}\to\wt{M},\omega\big)$ is automatically a Cartan geometry of type $(G,Q)$. As a next step, one considers the extended bundle $\wt{\mathcal{G}}:=\mathcal{G}\times_{Q}\wt{P}$ with respect to the homomorphism~$Q\to\wt{P}$. This is a principal bundle over $\wt{M}$ with structure group $\wt{P}$ and $j\colon \G\embed\wt{\G}$ denotes the natural inclusion. The equivariant extension of $\omega\in\Om^1(\G,\g)$ yields a unique Cartan connection $\wt{\omega}^{\rm ind}\in\Omega^1\big(\wt{\mathcal{G}},\tilde{\g}\big)$ of type $\big(\wt{G},\wt{P}\big)$ such that $ j^*\wt{\om}^{\rm ind}=i'\circ\om$. Altogether, one obtains a functor from parabolic geometries $(\calG\to M,\omega)$ of type $(G,P)$ to parabolic geometries $\big(\wt{\calG}\to\wt{M},\wt{\omega}^{\rm ind}\big)$ of type $\big(\wt{G},\wt{P}\big)$.

The relation between the corresponding curvatures is as follows: The previous assumptions yield a linear isomorphism $\tilde{\g}/\tilde{\p}\cong\g/\q$ and an obvious projection $\g/\q\to\g/\p$, where $\q\subseteq\p$ is the Lie algebra of $Q\subseteq P$. Composing these two maps one obtains a linear projection $\tilde{\g}/\tilde{\p}\to\g/\p$, whose dual map is denoted as $\ph\colon (\g/\p)^*\to(\tilde{\g}/\tilde{\p})^*$. Since $i'\colon \g\to\tilde{\g}$ is a homomorphism of Lie algebras, the curvature function $\wt{\kappa}^{\rm ind} \colon \wt{\mathcal{G}}\to\Lambda^2(\tilde{\g}/\tilde{\p})^*\otimes\tilde{\g}$ is related to $\kappa\colon \mathcal{G}\to\Lambda^2(\g/\p)^*\otimes\g$ by $\wt{\kappa}^{\rm ind}\circ j = (\La^2\ph\otimes i') \circ\kappa$. We note that $\wt{\kappa}^{\rm ind}$ is fully determined by this formula.

Since $i'$ is an embedding, the notation is in most cases simplif\/ied such that we write $\g\subseteq\tilde{\g}$, $\q=\g\cap\tilde{\p}$, etc.

\subsection{Algebraic setup and the homogeneous model}\label{construction}\label{homogen}
Here we specify the general setup for Fef\/ferman-type constructions from Section~\ref{setup} according to the description of oriented projective and conformal spin structures given in Sections~\ref{section-projective} and~\ref{section-conformalspin}, respectively. Let $\rr^{n+1,n+1}$ be the real vector space $\rr^{2n+2}$ with an inner product,~$h$, of split-signature. Let $\De_+^{n+1,n+1}$ and $\De_-^{n+1,n+1}$ be the irreducible spin representations of
\begin{gather*}
 \wt{G}:=\Spin(n+1,n+1)
\end{gather*}
as in Section~\ref{section-conformalspin}. We f\/ix two pure spinors $s_F\in\De_-^{n+1,n+1}$ and $s_E\in\De_{\pm}^{n+1,n+1}$ with non-trivial pairing, which is assigned for later use to be $ \langle s_E,s_F\rangle=-\frac12$. Note that $s_E$ lies in $\De_+^{n+1,n+1}$ if $n$ is even or in $\De_-^{n+1,n+1}$ if $n$ is odd.

Let us denote by $E,F\subseteq\rr^{n+1,n+1}$ the kernels of $s_E$, $s_F$ with respect to the Clif\/ford multiplication, i.e.,
\begin{gather*}
 E := \big\{ X \in \rr^{n+1,n+1}\colon X \cdot s_E = 0 \big\}, \qquad F := \big\{ X \in \rr^{n+1,n+1} \colon X \cdot s_F = 0 \big\}.
\end{gather*}
The purity of $s_E$ and $s_F$ means that $E$ and $F$ are maximally isotropic subspaces in $\rr^{n+1,n+1}$. The other assumptions guarantee that $E$ and $F$ are complementary and dual each other via the inner product~$h$. Hence we use the decomposition
\begin{gather}\label{decompR}
 \rr^{n+1,n+1}=E\oplus F \cong\rr^{n+1}\oplus{\rr^{n+1}}^*
\end{gather}
to identify the spinor representation $\De^{n+1,n+1}=\De_+^{n+1,n+1}\oplus\De_-^{n+1,n+1}$ with the exterior power algebra $\La^\bullet E \cong \La^\bullet \mbb{R}^{n+1}$, whose irreducible subrepresentations are $\De_-^{n+1,n+1} \cong \La^{\rm even} \mbb{R}^{n+1}$ and $\De_+^{n+1,n+1} \cong \La^{\rm odd} \mbb{R}^{n+1}$. When $n$ is even, respectively, odd, we can identify $\big(\Delta_-^{n+1,n+1}\big)^* \cong \Delta_+^{n+1,n+1}$, respectively $\big(\Delta_-^{n+1,n+1}\big)^* \cong \Delta_-^{n+1,n+1}$.

Now, let us consider the subgroup in $\wt{G}$ def\/ined by
\begin{gather*}
 G:=\{g\in \Spin(n+1,n+1)\colon g\cdot s_E=s_E,\, g\cdot s_F=s_F\}.
\end{gather*}
This subgroup preserves the decomposition \eqref{decompR} so that the restriction of the action to $F$ is dual to the restriction to $E$.
It further preserves the volume form on $E$, respectively $F\cong E^*$, which is determined by $s_E$ and $s_F$ according to the previous identif\/ications. Hence $G\cong\SL(n+1)$ and this def\/ines an embedding $i\colon \SL(n+1)\embed\Spin(n+1,n+1)$.\footnote{Instead of the embedding $\SL(n+1)\embed\Spin(n+1,n+1)$ we could also consider the embedding $\SL(n+1)\embed\SO(n+1,n+1)$. The advantage of employing the embedding into the spin group is two-fold: on the one hand, it is then seen directly that the induced conformal structure has a canonical spin structure, and, on the other hand, we can then use convenient spinorial objects for its characterisation.}

The $G$-invariant decomposition \eqref{decompR} determines a $G$-invariant skew-symmetric involution $K\in\so(n+1,n+1)$ acting by the identity on $E$ and minus the identity on~$F$. The relationship among $K$, $s_E$ and $s_F$ may be expressed as
\begin{gather} \label{eq-K-sE-sF}
 h (X , K(Y)) = -h(K(X),Y) = 2 \langle {s}_E , (X \wedge Y ) \cdot {s}_F \rangle ,
\end{gather}
where
 \begin{gather*}
 (X \wedge Y ) \cdot s_F = \frac{1}{2} ( X \cdot Y \cdot s_F - Y \cdot X \cdot s_F ) = X \cdot Y \cdot s_F + h( X ,Y) s_F .
\end{gather*}
The spin action of $\ti\g$ is denoted by $\bullet$, and thus $A\bullet s=-\frac14 A\cdot s$, for any $A\in\ti\g$ and $s\in\De$. In particular, $K\bullet s_F=-\frac12(n+1)s_F$ and $K\bullet s_E=\frac12(n+1)s_E$. Here we identify $\ti{\g}=\so(n+1,n+1)$ with $\La^2 \mbb{R}^{n+1,n+1}$. It is convenient to split $\ti\g$ in terms of irreducible $\g$-modules as
\begin{gather} \label{tig-decomp}
 \ti{\g}=\Lambda^2(E\oplus F)=\underbrace{(E\otimes F)_0}_{\mathfrak{g=\sl(n+1)}}\oplus \underbrace{(E\otimes F)_{Tr}\oplus \Lambda^2E\oplus\Lambda^2F}_{\mathfrak{g}^{\perp}} ,
\end{gather}
where $(E \otimes F)_{Tr} = \mbb{R} K$, and $K$ acts as $[ K , \phi ] = 2 \phi$, $[ K , \psi ] = -2 \psi$, $[ K , \lambda ] = 0$, for any $\phi \in \La^2 E$, $\psi \in \La^2 F$ and $\lambda \in E \otimes F$. Further, the annihilators of~$s_E$ and~$s_F$ in~$\tilde{\g}$ are the subalgebras $\ker s_E = \sl(n+1) \oplus \La^2 E$ and $\ker s_F = \sl(n+1) \oplus \La^2 F$.

The homogeneous model for conformal spin structures of signature $(n,n)$ is the space of isotropic rays in $\mathbb{R}^{n+1,n+1}$, $\wt{G}/\wt{P}\cong S^{n}\times S^{n}$. The subgroup $G\subseteq\widetilde{G}$ does not act transitively on that space. According to the decomposition \eqref{decompR}, there are three orbits: the set of rays contained in $E$, the set of rays contained in~$F$, and the set of isotropic rays that are neither contained in~$E$ nor in~$F$. Note that only the last orbit is open in~$\wt G/\wt P$, which is one of the requirements from Section~\ref{setup}. Therefore, we def\/ine $\wt P\subseteq \wt G$ to be the stabiliser of a ray through a light-like vector $\tilde{v}\in \rr^{n+1,n+1}\setminus(E\cup F)$. Denoting by $Q=i^{-1}(\wt{P})$ the stabiliser of the ray~$\mathbb{R}_{+}\tilde{v}$ in~$G$, we have the identif\/ication of $G/Q$ with the open orbit of the origin in~$\wt G/\wt P$. The subgroup $Q$, which is not parabolic, is contained in the parabolic subgroup $P\subseteq G$ def\/ined as the stabiliser in~$G$ of the ray through the projection of~$\tilde{v}$ to~$E$. In particular, $G/P$ is the standard projective sphere~$S^n$, the homogeneous model of oriented projective structures of dimension~$n$, and~$G/Q\to G/P$ is the canonical f\/ibration with the standard f\/ibre~$P/Q$, whose total space is the model Fef\/ferman space.

Let us denote by $L=\mathbb{R}\tilde{v}$ the line spanned by the light-like vector~$\tilde{v}$ and let~$L^{\perp}$ be the orthogonal complement in $\rr^{n+1,n+1}$ with respect to~$h$. The tangent space of $G/Q$ at the origin can be seen in three dif\/ferent ways, namely,
\begin{gather*}
 \big(L^{\perp}/L\big)[1] \cong \g/\q \cong \ti{\g}/\ti{\p}.
\end{gather*}
The latter isomorphism is induced by the embedding $\g\subseteq\ti{\g}$, the former one by the standard action of $\g\subseteq\ti{\g}$ on the vector $\tilde{v}\in\rr^{n+1,n+1}$. Both these identif\/ications are $Q$-equivariant.

There are several natural $Q$-invariant objects that in turn yield distinguished geometric objects on the general Fef\/ferman space. The $n$-dimensional $Q$-invariant subspace
\begin{gather*}
 f:=\big(\big(\bar{F}+L\big)/L\big)[1]\subseteq \big(L^\perp/L\big)[1], \qquad \text{where}\quad \bar{F}:=F\cap L^\perp,
\end{gather*}
which is isomorphic to $\p/\q\subseteq\g/\q$, the kernel of the projection $\g/\q\to\g/\p$. Another $n$-dimensional $Q$-invariant subspace is
\begin{gather*}
 e:=\big(\big(\bar{E}+L\big)/L\big)[1]\subseteq \big(L^\perp/L\big)[1], \qquad\text{where}\quad \bar{E}:=E\cap L^\perp.
\end{gather*}
The intersection $e\cap f$ is 1-dimensional with a distinguished $Q$-invariant generator that corresponds to the $G$-invariant involution $K\in\ti{\g}$,
\begin{gather*}
 k:=K+\ti\p \in\ti\g/\ti\p.
\end{gather*}
Note that all these objects are isotropic with respect to the natural conformal class induced by the restriction of~$h$ to~$L^\perp\subseteq\rr^{n+1.n+1}$. In particular, both $e$ and $f$ are maximally isotropic subspaces such that
\begin{gather}\label{flag}
 k \in e\cap f \subseteq k^{\perp} =e + f.
\end{gather}

In Section~\ref{setup} we introduced a map $\ph\colon (\g/\p)^*\to(\tilde{\g}/\tilde{\p})^*$, the dual map to the projection $\ti\g/\ti\p\cong\g/\q\to\g/\p$. The kernel of this projection is just $f$ and the image of $\ph$ is identif\/ied with its annihilator, which will be denoted by~$\f$. Since $f$ is a maximally isotropic subspace in $\ti\g/\ti\p\cong\g/\q$,
\begin{gather*}
\f\cong f[-2].
\end{gather*}
Since $(\ti{\g}/\ti{\p})^*\cong\ti{\p}_+$, we may conclude with the help of explicit matrix realisations from Appendix~\ref{appendixB} that $ \f=\ti{\p}_+\cap\ker s_F$. Moreover, we note that
\begin{gather}\label{eq-tp+p}
\big( \tilde{\p}_+ \cap \ker s_F \big)_{E \otimes F}= \p_+,\qquad \big( \tilde{\p} \cap \ker s_F \big)_{E \otimes F} = \p , \\ \label{useful}
\La^2F\cap\ti\p = \La^2\bar{F} \subseteq \ti{\g}_0,\qquad \big[\ti{\p}_+,\La^2\bar{F}\big]=\f,\qquad \big[\f,\La^2\bar{F}\big]=0.
\end{gather}

\subsection{The Fef\/ferman space and induced structure}\label{sec-feff}
The pairs of Lie groups $(G,P)$ and $\big(\wt{G},\wt{P}\big)$ from the previous subsection satisfy all the properties to launch the Fef\/ferman-type construction.
\begin{Proposition}\label{prop-Feffspace} The Fefferman-type construction for the pairs of Lie groups $(G,P)$ and $\big(\widetilde{G},\widetilde{P}\big)$ yields a natural construction of conformal spin structures $\big(\wt{M},\mb{c}\big)$ of signature $(n,n)$ from $n$-dimensional oriented projective structures $(M,\mb{p})$. The Fefferman space $\wt{M}$ is identified with the total space of the weighted cotangent bundle without the zero section $T^*M(2)\backslash \{0\}$.
\end{Proposition}
\begin{proof} The f\/irst part of the statement is obvious from the general setting for Fef\/ferman-type constructions and the Cartan-geometric description of oriented projective and conformal spin structures.

The second part is shown due to two natural identif\/ications: On the one hand, the Fef\/ferman space is by~\eqref{Mtilde} equal to the total space of the associated bundle $\wt{M}\cong\calG\times_{P} P/Q$ over $M$. On the other hand, the weighted cotangent bundle to $M$ is identif\/ied with the associated bundle $T^*M(2)\cong\G\x_P(\g/\p)^*(2)$ with respect to action of $P$ induced by the adjoint action and the representation~\eqref{projectivedens} for $w=2$. Hence it remains to verify that the action of $P$ on $(\g/\p)^*(2)\setminus\{0\}$ is transitive and $Q$ is a stabiliser of a non-zero element. But this is a purely algebraic task, which may be easily checked in a concrete matrix realisation.
\end{proof}

From the algebraic setup in Section~\ref{construction} we easily conclude number of specif\/ic features of the induced conformal structure on $\wt{M}$:
\begin{Proposition}\label{prop-objects}
The conformal spin structure $\big(\wt{M},\mb{c}\big)$ induced from an oriented projective structure $(M,\mb{p})$ by the Fefferman-type construction admits the following tractorial objects that are all parallel with respect to the induced tractor connection:
\begin{enumerate}\itemsep=0pt
\item[$(a)$] pure tractor spinors $\mb{s}_{E}\in\Ga\big(\wt{\mc{S}}_{\pm}\big)$ and $\mb{s}_{F}\in\Ga\big(\wt{\mc{S}}_{-}\big)$ with non-trivial pairing,
\item[$(b)$] a tractor endomorphism $\mb{K}\in\Ga\big(\mc{A}\wt{M}\big)$ which is an involution, i.e., $\mb{K}^2=\one_{\wt{\mc{T}}}$, and which acts by the identity, respectively minus the identity on the maximally isotropic complementary subbundles $\wt{\mc{E}}:=\ker\mb{s}_E$, respectively $\wt{\mc{F}}:=\ker\mb{s}_F$ of $\wt{\mc{T}}$.
\end{enumerate}
The corresponding underlying objects $\eta=\Pi_0^{\wt{\mc{S}}}(\mb{s}_E)$, $\chi=\Pi_0^{\wt{\mc{S}}}(\mb{s}_F)$ and $k=\Pi_0^{\mc{A}\wt{M}}(\mb{K})$ satisfy:
\begin{enumerate}\itemsep=0pt %\setcounter{enumi}{2}
\item[$(c)$] $\eta\in\Ga\big(\wt{\Sigma}_\pm\big[\frac12\big]\big)$ and $\chi\in\Ga\big(\wt{\Sigma}_-\big[\frac12\big]\big)$ are pure spinors, whose kernels $\wt{e}:=\ker\eta$ and $\wt{f}:=\ker\chi$ have $1$-dimensional intersection and~$\wt{f}$ coincides with the vertical subbundle of $\wt{M}\to M$,
\item[$(d)$] $k\in\Ga\big(T\wt{M}\big)$ is a nowhere-vanishing light-like vector field generating the intersection $\wt{e}\cap\wt{f}$.
\end{enumerate}
\end{Proposition}

\begin{proof}The $G$-invariant spinor $s_E\in\De_\pm$ gives rise to the tractor spinor $\mb{s}_E\in\Ga\big(\wt{\mc{S}}_\pm=\G\x_Q \De_\pm\big)$ such that it corresponds to the constant ($Q$-equivariant) map $\G\to\De_\pm$. Hence $\mb{s}_E$ is automatically parallel with respect to the induced tractor connection on $\wt{\mc{S}}_\pm$. Similar reasoning for other $G$-invariant objects and their compatibility described above yield the f\/irst part of the statement. In particular, $\wt{\mc{E}}=\G\x_Q E$, $\wt{\mc{F}}=\G\x_Q F$ and the decomposition $\wt{\mc{T}}=\wt{\mc{E}}\oplus\wt{\mc{F}}$ corresponds to the decomposition~\eqref{decompR}.

The f\/iltration $L\subseteq L^\perp\subseteq\rr^{n+1,n+1}$ gives rise to the f\/iltration of the standard tractor bundle, which can be written as
\begin{gather*}
\begin{pmatrix}
\wt{\mbb{E}}[-1]\\ 0\\ 0
\end{pmatrix}
\subseteq
\begin{pmatrix}
\wt{\mbb{E}}[-1]\\ \wt{\mbb{E}}_a[1] \\ 0
\end{pmatrix}
\subseteq
\begin{pmatrix}
\wt{\mbb{E}}[-1]\\
\wt{\mbb{E}}_a[1] \\
\wt{\mbb{E}}[1]
\end{pmatrix} = \wt{\mc{T}}.
\end{gather*}
In particular, the subbundles associated to $\bar E,\bar F\subseteq L^\perp$ are distinguished by the middle slot. The corresponding $Q$-invariant maximally isotropic subspaces $e,f\subseteq\g/\q$ determine the distributions $\G\x_Q e$ and $\G\x_Q f$ in $T\wt{M}=\G\x_Q\g/\q$. According to the tractor Clif\/ford action~\eqref{traCli} it follows that these are precisely the kernels of the spinors~$\eta$ and~$\chi$. Since these subspaces are maximally isotropic, the corresponding spinors are pure. Since $f\cong\p/\q$ is the kernel of the projection $\g/\q\to\g/\p$, the corresponding subbundle~$\wt{f}$ is identif\/ied with the vertical subbundle of the projection $\wt{M}\to M$. The intersection $e\cap f$ is 1-dimensional and it is generated by the projection of $K\in\ti\g$ to $\ti\g/\ti\p$. Indeed, $K$ cannot be contained in $\ti{\p}$, since $K$ acts by the identity on $E$ and minus the identity on~$F$ and~$\ti\p$ is the stabiliser of a line that is neither contained in $E$ nor in $F$. Altogether, the corresponding vector f\/ield~$k$ on $\wt{M}$ is a nowhere-vanishing generator of~$\wt{e}\cap\wt{f}$, in particular, it is light-like.
\end{proof}

\subsection{Relating tractors, Weyl structures and scales}\label{Relpartra}
As a technical preliminary for further study we now relate natural objects associated to the original projective Cartan geometry $(\G,\om)$ on $M$ and the induced conformal geometry $(\wt\G,\wt{\om}^{\rm ind})$ on the Fef\/ferman space $\wt M$.

Since $G\subseteq\widetilde{G}$, any $\widetilde{G}$-representation $V$ is also a $G$-representation, which yields compatible tractor bundles over $M$ and $\wt M$ with compatible tractor connections: $\mathcal{V}=\mathcal{G}\times_{P} V\to M$ with the tractor connection~$\nabla$ induced by~$\om$ and $\widetilde{\mathcal{V}}=\widetilde{\mathcal{G}}\times_{\wt{P}} V=\mathcal{G}\times_{Q}V\to \widetilde{M}$ with the tractor connection~$\wt\nabla^{\rm ind}$ induced by $\wt\om^{\rm ind}$. Sections of $\mathcal{V}$ bijectively correspond to $P$-equivariant functions $\ph\colon \mathcal{G}\to V$, while sections of $\wt{\mathcal{V}}$ correspond to $Q$-equivariant functions $\ph\colon \mathcal{G}\to V$.
Since $Q\subseteq P$, every section of $\mathcal{V}$ gives rise to a section of $\wt{\mathcal{V}}$, and we can view $\Gamma(\mathcal{V})\subseteq\Gamma\big(\wt{\mathcal{V}}\big)$. Now, Proposition~3.2 in~\cite{cap-gover-cr-tractors} admits a straightforward generalisation to Fef\/ferman-type constructions for which $P/Q$ is connected and thus, in particular, to the one studied in this article:

\begin{Proposition} \label{mini}\quad
\begin{enumerate}\itemsep=0pt
\item[$(a)$] A section $s\in\Gamma\big(\widetilde{\mathcal{V}}\big)$ is contained in $\Gamma(\mathcal{V})$ $($i.e., the corresponding $Q$-equivariant function $\ph$ is indeed $P$-equivariant$)$ if and only if $\wt\nabla^{\rm ind} s$ is strictly horizontal $($i.e., $v^a\wt\nabla^{\rm ind}_a s=0$ for all $v^a\in\Ga\big(\wt f\big))$.
\item[$(b)$] The restriction of $\wt\nabla^{\rm ind}$ to $\Gamma(\mathcal{V})\subseteq\Gamma\big(\wt{\mathcal{V}}\big)$ coincides with the tractor connection~$\nabla$.
\end{enumerate}
\end{Proposition}

\begin{Remark}\label{rem-Weyl} Another instance of compatible bundles over $M$ and $\wt M$ is provided by the density bundles $\mbb{E}(w)$ and $\wt{\mbb{E}}[w]$, which are def\/ined via the representation of $P$ and $\wt{P}$ as in~\eqref{projectivedens} and~\eqref{conformaldens}, respectively. Restricting these representations to~$Q$, it easily follows that the notation is indeed compatible so that we can view $\Ga(\mbb{E}(w))\subseteq\Ga\big(\wt{\mbb{E}}[w]\big)$.

Both projective and conformal density bundles can be described as associated bundles to the respective bundles of scales. Hence everywhere positive sections of density bundles are considered as scales. In particular, the inclusion $\Ga(\mbb{E}_+(1))\subseteq\Ga\big(\wt{\mbb{E}}_+[1]\big)$ may be interpreted so that any projective scale induces a conformal one. Such conformal scales will be called~\emph{reduced scales}. An intrinsic characterisation of reduced scales among all conformal ones is formulated in Proposition~\ref{prop-nnscales}.
\end{Remark}

The previous remark yields that any projective exact Weyl structure on $M$ induces a conformal exact Weyl structure on $\wt M$. This fact can be generalised as follows:

\begin{Proposition}\label{prop-Weyl}
Any projective $($exact$)$ Weyl structure on $M$ induces a conformal $($exact$)$ Weyl structure on the Fefferman space~$\wt M$.
\end{Proposition}
\begin{proof} A version of this result in a more general context was proved in \cite[Proposition~6.1]{jesse-quaternionic}: any Weyl structure for $\om$ induces a Weyl structure for $\wt\om^{\rm ind}$ if $P_+\subseteq\wt P$ and $\big(G_0\cap\wt P\big)\subseteq\wt G_0$. But both these conditions are satisf\/ied as follows from the setup in Section~\ref{construction} and explicit realisations in Appendix~\ref{appendixB}.
\end{proof}

Conformal Weyl structures induced by projective ones as above will be called \emph{reduced Weyl structures}.

\subsection{Normality}\label{normality}
Here we show that our Fef\/ferman-type construction does not preserve the normality in general, see Proposition~\ref{prop-nonnormal}.
This can be shown directly as we did in a previous version of the article, see \href{https://arxiv.org/abs/1510.03337v2}{arXiv:1510.03337v2}. %\cite{hsstz}.
Alternatively, we can treat the construction as the composition of two other constructions via a natural intermediate Lagrangean contact structure.

A \emph{Lagrangean contact structure} on $M'$ consists of a contact distribution $\mathcal{H}\subseteq TM'$ together with a decomposition $\mathcal{H}=e' \oplus f'$ into two subbundles that are maximally isotropic with respect to the Levi form $\mc{H}\x\mc{H}\to TM'/\mc{H}$.
Such structure on a manifold $M'$ of dimension $2n-1$ is equivalently encoded as a normal parabolic geometry of type~$(G,P')$, where $G=\SL(n+1)$
and $P'\subseteq G$ is the stabiliser of a f\/lag of type line-hyperplane in the standard representation~$\rr^{n+1}$. For $n>2$ there are three harmonic curvatures, two of which are torsions whose vanishing is equivalent to the integrability of the respective subbundles $e',f'\subseteq\mc{H}$. For $n=2$ there are two harmonic curvatures of homogeneity~4, hence the Cartan connection is torsion-free. In that case both~$e'$ and $f'$ are 1-dimensional and thus automatically integrable.

On the one hand, $P'$ is contained in $P$, where $P\subseteq G$ is the stabiliser of a ray in $\rr^{n+1}$. For suitable choices as in Appendix~\ref{appendixB}, the Lie algebra to~$P'$ consists of matrices of the form
\begin{gather*}
\p'=\begin{pmatrix}
a&U^t&w\\
0 & B &V\\
0&0& c
\end{pmatrix}.
\end{gather*}
Given a projective Cartan geometry $(\G\to M,\om)$ of type $(G,P)$, it turns out that the correspondence space $M':=\G/P'$ can be identif\/ied with the projectivised cotangent bundle $\mc{P}(T^*M)$. The Cartan geometry $(\G\to M',\om)$ of type $(G,P')$ is regular and thus it covers a natural Lagrangean contact structure on~$M'$. In particular, the canonical contact distribution on $\mc{P}(T^*M)$ coincides with~$\mc{H}$ and
the vertical subbundle of the projection $M'\to M$ coincides with one of the two distinguished subbundles, say $f'\subseteq\mc{H}$. As in general, this construction preserves normality. In accord with~\cite{cap-correspondence}, respectively \cite[Section~4.4.2]{cap-slovak-book} we may state:
\begin{Proposition} Let $(\mathcal{G}\to M,\omega)$ be a normal projective parabolic geometry and let $(\G\to M',\omega)$ be the corresponding normal Lagrangean contact parabolic geometry. The latter geometry is torsion-free if and only if $n=2$ or it is flat, i.e., the initial projective structure is flat.
\end{Proposition}

On the other hand, $P'$ contains $Q$, where $Q=G\cap\wt{P}$ as before. This allows us to consider the Fef\/ferman-type construction for the pairs $(G,P')$ and $(\wt{G},\wt{P})$. Given a Lagrangean contact structure on~$M'$, it induces a conformal spin structure on $\wt{M}=\G/Q$. This construction is indeed very similar to the original Fef\/ferman construction; one deals with dif\/ferent real forms of the same complex Lie groups in the two cases. That is why the following statement and its proof is analogous to the one for the CR case. Following \cite{cap-gover-cr-tractors}, respectively \cite[Section~4.5.2]{cap-slovak-book} we may state:
\begin{Proposition}\label{PropLagNor} Let $(\G\to M',\omega)$ be the normal Lagrangean contact parabolic geometry and let $\big(\wt{\G}\to\wt M,\wt{\om}^{\rm ind}\big)$ be the conformal parabolic geometry obtained by the Fefferman-type construction. Then $\wt{\om}^{\rm ind}$ is normal if and only if $\om$ is torsion-free.
\end{Proposition}

Altogether, composing the two previous steps we obtain our projective-to-conformal Fef\/fer\-man-type construction with the desired control of the normality. Note that from \eqref{flag} and the respective matrix realisations it follows that the induced objects on $\wt{M}=T^*M(2)\setminus\{0\}$ from Proposition~\ref{prop-objects} correspond to the induced objects on $M'=\mc{P}(T^*M)$. In particular, the vertical subbundle of the projection $\wt{M}\to M'$ is spanned by $k$ and the decomposition $k^\perp =\wt{e}\oplus\wt{f}\subseteq T\wt{M}$ descends to the decomposition $\mc{H}=e'\oplus f'\subseteq TM'$
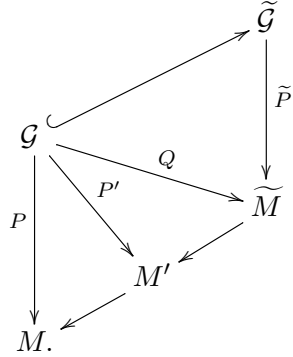
$$
\xymatrix@R=.8\baselineskip{
& & \wt\G \ar[ddd]^(.4){\wt P} \\
& & \\
\G\ \ar@{^{(}->}[uurr] \ar[drr]^(.55)Q \ar[ddr]^(.46){P'} \ar[ddd]_(.4)P & & \\
& & \wt M \ar[dl] \\
& M' \ar[dl] & \\
M. & & \\
}
$$

\begin{Proposition}\label{prop-nonnormal}Let $(\mathcal{G}\to M,\omega)$ be a normal projective parabolic geometry and let $\big(\widetilde{\mathcal{G}}\to\wt M,\widetilde{\omega}^{\rm ind}\big)$ be the conformal parabolic geometry obtained by the Fefferman-type construction.
\begin{enumerate}\itemsep=0pt
 \item[$(a)$] If $\dim M=2$ then $\wt{\om}^{\rm ind}$ is normal.
 \item[$(b)$] If $\dim M>2$ then $\wt{\om}^{\rm ind}$ is normal if and only if $\om$ is flat.
\end{enumerate}
Moreover, independently of the dimension of $M$, $\wt{\om}^{\rm ind}$ is flat if and only if $\om$ is flat.
\end{Proposition}

\subsection{Remarks on torsion-free Lagrangean contact structures}\label{sec-Lagrange-contact}
At this stage it is easy to formulate a local characterisation of split-signature conformal structures arising from torsion-free Lagrangean contact structures, see Proposition~\ref{PropLagChar}. As before, the results and their proofs are very analogous to those in the CR case, therefore we just quickly indicate the reasoning and point to dif\/ferences.

As in Proposition~\ref{prop-objects}, the $G$-invariant algebraic objects induce the tractor f\/ields $\mb{s}_E$, $\mb{s}_F$ and $\mb{K}$ on the conformal Fef\/ferman space that are parallel with respect to the induced tractor connection and have the required compatibility properties. But, starting with a torsion-free Lagrangean contact structure, the induced connection is already normal. In particular, the corresponding underlying objects $\chi$, $\eta$ and $k$ are pure twistor spinors and a light-like conformal Killing f\/ield, respectively.

The existence of parallel tractors $\mb{s}_E$, $\mb{s}_F$ and $\mb{K}$ with the algebraic properties as in Proposition~\ref{prop-objects} are by no means independent conditions:

\begin{Proposition}\label{PropLagCond} Let $\big(\wt M,\mb{c}\big)$ be a conformal spin structure of split-signature $(n,n)$. Then the following conditions are locally equivalent:
\begin{enumerate}\itemsep=0pt
\item[$(a)$] The spin tractor bundle admits two pure parallel tractor spinors $\mb{s}_E\in\Gamma\big(\wt{\mc{S}}_{\pm}\big)$ and $\mb{s}_F\in\Gamma\big(\widetilde{\mc{S}}_-\big)$ with non-trivial pairing.
\item[$(b)$] The conformal holonomy $\Hol(\mb{c})$ reduces to $\SL(n+1)\subseteq\Spin(n+1,n+1)$ preserving a~decomposition into maximally isotropic subspaces $E\oplus F=\mathbb{R}^{n+1,n+1}$.
\item[$(c)$] The adjoint tractor bundle admits a parallel involution $\mathbf{K}\in\Gamma\big(\mathcal{A}\wt{M}\big)$, i.e., $\mathbf{K}^2=\id_{\wt{\mc T}}$.
\end{enumerate}
\end{Proposition}

The only subtle point within the proof concerns the consequences of property (c). The existence of a parallel skew-symmetric involution $\mb{K}$ on the standard tractor bundle immediately implies that the conformal holonomy $\Hol(\mb{c})$ is reduced to $\GL(n+1)$. But, analogously to the corresponding discussion for the CR case in~\cite{cap-gover-cr} or~\cite{felipe-unitaryholonomy}, one can show that $\Hol(\mb{c})$ is actually contained in $\SL(n+1)$. The rest follows easily.

It turns out that conformal spin structures induced by torsion-free Lagrangean contact structures are locally characterised by any of the three equivalent conditions above. Indeed, according to results from~\cite{CGH-holonomy}, the holonomy reduction of the conformal structure to $G=\SL(n+1)\subseteq\Spin(n+1,n+1)=\wt G$ yields the so-called curved orbit decomposition of~$\wt M$, which corresponds to the decomposition of the homogeneous model $\wt G/\wt P$ with respect to the action of $G$. Each subset from the decomposition of~$\wt M$, provided it is non-empty, further carries a geometry of the same type as its counterpart in the homogeneous model. From Section~\ref{homogen} we know there is one open and two closed $n$-dimensional orbits. The closed $n$-dimensional orbits carry Cartan geometries of type~$(G,P)$, and thus inherit projective structures, the open orbit carries a Cartan geometry of type $(G,Q)$. Note that the two closed orbits coincide with the zero sets of~$\chi$ and~$\eta$, the open subset is the one where both spinors, and thus $k$, are non-vanishing. Since $k$ is the conformal Killing f\/ield corresponding to the parallel adjoint tractor $\mb{K}$, it inserts trivially into the curvature of the normal Cartan connection, cf.~\eqref{cKf}. Hence, according to~\cite{cap-correspondence}, the Cartan geometry of type $(G,Q)$ on the open orbit of~$\wt M$ descends to a Cartan geometry of type $(G,P')$ on the local leaf space~$M'$ determined by~$k$. It follows that this Cartan geometry is torsion-free and thus determines a torsion-free Lagrangean contact structure. Altogether, following~\cite{cap-gover-cr} we may state the following characterisation:

\begin{Proposition}\label{PropLagChar}A split-signature conformal spin structure is locally induced by a torsion-free Lagrangean contact structure via the Fefferman-type construction if and only if any of the equivalent conditions from Proposition~{\rm \ref{PropLagCond}} holds and the underlying twistor spinors $\chi$ and $\eta$ and the conformal Killing field $k$ are nowhere-vanishing.
\end{Proposition}

\subsection[The exceptional case: dimension $n=2$]{The exceptional case: dimension $\boldsymbol{n=2}$} \label{sec-dim2}
From Section~\ref{normality} we know that the intermediate 3-dimensional Lagrangean contact structure on $M'$ induced by a 2-dimensional projective structure on $M$ is torsion-free. Hence the induced conformal Cartan geometry on $\wt M$ is normal and thus all the equivalent conditions from Proposition \ref{PropLagCond} are satisf\/ied. Moreover, the fact that it comes from a projective structure implies that any vertical vector of the projection $\wt M\to M$ inserts trivially into the Cartan curvature, i.e.,
\begin{gather}\label{vertcurv}
\io_{X}\wt{\ka}(u)=0, \qquad \text{for all $X\in f$, $u\in\wt\G$}.
\end{gather}
Analogously to the discussion before Proposition \ref{PropLagChar} we may conclude:

\begin{Proposition} A conformal spin structure of signature $(2,2)$ is locally induced by a $2$-dimensional projective structure via the Fefferman-type construction if and only if any of the equivalent conditions from Proposition~{\rm \ref{PropLagCond}} holds, the underlying twistor spinors $\chi$ and $\eta$ and the conformal Killing field $k$ are nowhere-vanishing and the curvature of the normal conformal Cartan connection satisfies~\eqref{vertcurv}.
\end{Proposition}

\begin{Remark}Conformal structures induced from $2$-dimensional projective structures are well-studied, see, e.g., \cite{crampin-saunders, dunajski-tod, nurowski-sparling-cr}. Notably, the intermediate $3$-dimensional Lagrangean contact structure can be equivalently viewed as a path geometry (or the geometry associated to second order ODEs modulo point transformations). Such structure is induced by a projective structure (i.e., the paths are the unparametrised geodesics of the projective class of connections) if and only if one of the two harmonic curvatures vanishes. It follows from~\cite{nurowski-sparling-cr} that this is equivalent to vanishing of the self-dual, respectively anti-self-dual part of the Weyl curvature of the induced conformal structure. In particular, the condition~\eqref{vertcurv} in the previous proposition can be replaced by the condition that the conformal structure is half-f\/lat.
\end{Remark}

\section{Normalisation and characterisation}\label{nor&char}
By Proposition \ref{prop-nonnormal}, for $n\geq 3$, the induced conformal Cartan connection associated to a non-f\/lat $n$-dimensional projective structure dif\/fers from the normal conformal Cartan connection for the induced conformal structure. In this section we will analyse the form of the dif\/ference and thus derive properties of the induced conformal structures. Furthermore, we will show that any split-signature conformal manifold having these properties is locally equivalent to the conformal structure on the Fef\/ferman space over a projective manifold.

\subsection{The normalisation process} \label{nprocess}
We are going to normalise the conformal Cartan connection $\wt\om^{\rm ind}\in\Om^1\big(\wt\G,\ti\g\big)$ that is induced by a~normal projective Cartan connection $\om\in\Om^1(\G,\g)$. Any other conformal Cartan connection~$\wt\om'$ dif\/fers from~$\wt\om^{\rm ind}$ by some $\Psi\in\Om^1\big(\wt\G,\ti\g\big)$ so that $\wt\om'=\wt\om^{\rm ind}+\Psi$. This $\Psi$ must vanish on vertical f\/ields and be $\wt P$-equivariant. The condition on ${\wt\om}'$ to induce the same conformal structure on $\wt M$ as $\wt\om^{\rm ind}$ is that $\Psi$ has values in $\ti\p\subseteq\ti\g$. One can therefore regard $\Psi$ as a $\wt P$-equivariant function $\Psi\colon \wt\G\goesto (\ti\g/\ti\p)^*\t\ti\p$. According to the general theory as outlined in \cite[Section~3.1.13]{cap-slovak-book} there is a unique such $\Psi$ such that the curvature function $\wt \kappa'$ of $\wt\om'$ satisf\/ies $\wt{\partial}^* \wt \ka'=0$, and then $\wt\om'$ is the normal conformal Cartan connection $\wt\om^{\rm nor}$.

The failure of $\wt\om^{\rm ind}$ to be normal is given by $\wt{\partial}^*\wt \ka^{\rm ind}\colon \wt\G\goesto (\ti\g/\ti\p)^*\t\ti\p$. The normalisation of $\wt\om^{\rm ind}$ proceeds by homogeneity of $(\ti\g/\ti\p)^*\t\ti\p$, which decomposes into two homogeneous components according to the decomposition $\ti\p=\ti\g_0\oplus\ti\p_{+}$. In the f\/irst step of normalisation one looks for a~${\Psi}^{1}$ such that ${\wt\om}^{1}=\wt\om+{\Psi}^{1}$ has $\wt{\partial}^*\wt\ka^1$ taking values in the highest homogeneity, i.e., $\wt{\partial}^*\wt\ka^1\colon \wt\G\goesto (\ti\g/\ti\p)^*\t\ti\p_+$.

To write down this f\/irst normalisation we employ Weyl structures $\wt\G_0{\embed}\wt\G$. By Proposition~\ref{prop-Weyl} we can take a reduced Weyl structure, i.e., one that is induced by a reduction $ \G_0{\embed}\G\embed\wt \G$ with respect to the structure group $Q_0:=Q\cap G_0$. This allows us to project ${\wt\partial}^*\wt \ka^{\rm ind}$ to $\big({\wt\partial}^*\wt\ka^{\rm ind}\big)_{0}\colon \G_0\goesto (\ti\g/\ti\p)^*\t\ti\g_0$ and to employ the $\wt G_0$-equivariant Kostant Laplacian $\wt\lapl\colon (\ti\g/\ti\p)^*\t\ti\g_0\goesto (\ti\g/\ti\p)^*\t\ti\g_0$, $\wt\lapl:=\wt{\partial}\circ\wt{\partial}^*+\wt{\partial}^*\circ\wt{\partial}$. For the f\/irst normalisation step we need to form a map ${\Psi}^{1}\colon \wt\G\goesto (\ti\g/\ti\p)^*\t\ti\p$ that agrees with $-\wt\lapl^{-1}\big(\wt\delstar\wt\ka^{\rm ind}\big)_{0}$ in the $\ti\g_0$-component. If we have formed any such ${\Psi}^{1}$ along $\G_0{\embed}\wt\G$ we can just equivariantly extend this to all of~$\wt\G$.

To proceed with the analysis of the normalisation we need to establish a couple of technical lemmas. As before, we denote by $\f\subset\ti\p_+\cong(\ti\g/\ti\p)^*$ the annihilator of $f=\p/\q\subset\g/\q\cong\ti\g/\ti\p$. Recall that $\f=\varphi(\p_+)\cong f[-2]$.

\begin{Lemma} \label{forJosef} Let $V$ be a $\g$-representation contained in a $\ti{\g}$-representation $\wt{V}$ and denote by $\phi\mapsto\wt\phi$ the inclusion $\Lambda^k\p_+\otimes V \hookrightarrow \Lambda^k\ti{\p}_+ \otimes \wt{V}$ induced by $\varphi\colon \p_+\to\ti\p_+$ and $V\embed\wt{V}$. Then, for any $\phi\in\Lambda^{k}\mathbb{\p}_+\otimes V$,
\begin{gather*}
\wt{\partial^*\phi}-{\wt{\partial}}^*\wt{\phi} \in\Lambda^{k-1}\f\otimes\big(\Lambda^2 \bar{F}\bullet V\big)\subseteq \Lambda^{k-1}\ti{\p}_+\otimes \wt{V}.
\end{gather*}
In particular, for the adjoint representations, $\delstar\phi=0$ if and only if $\wt{\partial}^*\wt\phi \in \Lambda^{k-1}\f\otimes\Lambda^2 \bar{F}$.
\end{Lemma}

\begin{proof}For the sake of presentation, assume that $\phi$ is decomposable, i.e., of the form $\phi=Z_1\wedge\dots\wedge Z_k\otimes v$, where $Z_i\in\p_+$ and $v\in V$. Let us denote by the same symbols also the images of these elements under the inclusion $\g\embed\ti\g$ and $V\embed\wt{V}$, i.e., $Z_i\in\ti\p$ and $v\in\wt{V}$, respectively. Let $\wt{Z}_i\in \f$ be the images of $Z_i$ under the inclusion $\varphi\colon \p_+\to\ti\p_+$. Now, by def\/inition of the Kostant co-dif\/ferential, the dif\/ference $\wt{\partial^*\phi}-{\wt{\partial}}^*\wt{\phi}$
evaluated on any $k-1$ elements from~$\ti\g/\ti\p$ is a~linear combination of terms of the form
\begin{gather}\label{eqforJosef}
\big(Z_i - \wt Z_i\big)\bullet v.
\end{gather}
However, the dif\/ferences $Z_i-\wt Z_i\in\ti\p$ are represented by the matrices as in~\eqref{parabolic} in the Appendix where only the $Z$-entries are non-vanishing and hence contained in $\Lambda^2 F\cap \ti\p=\Lambda^2 \bar{F}$. Thus~\eqref{eqforJosef} belong to the image of $\bullet \colon \Lambda^2 \bar{F} \times V\to \wt{V}$ and the f\/irst claim follows.

For the second claim we use that $\La^2 \bar{F} \bullet \g = \big[\La^2 \bar{F},\g\big] \subseteq \Lambda^2 F$ and $\La^2 F\cap\g=0$: since $\wt{\partial^*\phi}$ (evaluated on any $k-1$ elements from $\ti\g/\ti\p$) has values in $\g\subset\ti\g$, vanishing of $\delstar\phi$ is equivalent to $\wt{\partial}^*\wt\phi$ having values in $\La^2 F$. But ${\wt{\partial}}^*\wt{\phi}$ has generally values in $\ti\p$ and $\Lambda^2 F\cap \ti\p=\Lambda^2 \bar{F}$, hence the claim follows.
\end{proof}

\begin{Lemma}\label{lem-hor} If $\psi\in \ti\p_+\wedge\f \t \La^2\bar{F} \subseteq \La^2\ti\p_+\t\ti\p$ then $\wt{\partial}^*\psi\in \f\t \f \subseteq\ti\p_+\t\ti\p_+$.
\end{Lemma}

\begin{proof} $\psi$ is a sum of terms of the form $Z_1\wedge Z_2\otimes A$, where $Z_1\in\ti\p_+$, $Z_2\in\f$ and $A\in\La^2\bar{F}$. Applying the Kostant co-dif\/ferential gives
\begin{gather*}
\wt{\partial}^*(Z_1\wedge Z_2\otimes A)=Z_1\otimes [Z_2,A] - Z_2\otimes [Z_1,A].
\end{gather*}
Now $[Z_2,A]$ belongs to $\big[\f,\Lambda^2\bar{F}\big]=0$ and $[Z_1,A]$ belongs to $\big[\ti\p_+,\Lambda^2\bar{F}\big]=\f $, hence the claim follows.
\end{proof}

The following lemma contains the crucial information which is necessary to perform our normalisation. We are going to specify the curvature function $\wt\ka^{\rm ind}$ (later also $\wt\ka^{\rm nor}$) by describing its values along the natural $Q$-reduction $\G\embed\wt\G$ over~$\wt M$. Recall from Section~\ref{homogen} that $\Lambda^2\bar{F}$ is a~$Q$-invariant subspace in $\ti\g_0$, which can be identif\/ied with $(\Lambda^2f)[-2]$.

\begin{Lemma}\label{lem-psi} For any $u\in\G$, we have
\begin{gather*}
 \wt{\partial}^*\wt{\kappa}^{\rm ind}(u)\in \f\otimes\Lambda^2\bar{F} \subseteq\ti\p_+\t\ti\g_0.
\end{gather*}
Identifying $\Lambda^2\bar{F}\cong \big(\Lambda^2 f\big)[-2]$ and $\f\cong f[-2]$, we have in fact $\wt{\partial}^*\wt{\kappa}^{\rm ind}(u) \in \big(f \odot \Lambda^2 f\big) [-4] $, i.e., $\wt{\partial}^*\wt{\kappa}^{\rm ind}(u)$ is contained in the kernel of the alternation map
\begin{gather*}
\operatorname{alt}\colon \ \big(f\otimes\Lambda^2 f\big) [-4]\to\big(\Lambda^3 f\big)[-4].
\end{gather*}
\end{Lemma}
\begin{proof} It is a general assumption that $\wt\om^{\rm ind}$ is induced by a normal projective Cartan connection on~$\G$, i.e., $\delstar\ka(u)=0$, for any $u\in\G$. Hence it follows from Lemma~\ref{forJosef} that $\wt{\partial}^*\wt{\kappa}^{\rm ind}(u)$ belongs to $\f\otimes\Lambda^2\bar{F}\cong \big(f\otimes\Lambda^2 f\big)[-4]$.

Further we need a f\/iner discussion involving the properties of $\ka\colon \G\to\La^2\p_+\t\g$ to show that~$\ka(u)$ belongs to the kernel of the $Q$-equivariant map $\La^2\p_+\t\g\to\big(\La^3 f\big)[-4]$ given by
\begin{gather}\label{composition}
\phi\mapsto \operatorname{alt}\big(\wt{\partial}^*\wt{\phi}\big).
\end{gather}
Note that any element $\phi\in\La^2 \p_+\t\p_+$ for which ${\partial}^*\phi=0$ is mapped to zero: since $\wt\phi\in\La^2\ti\p_+\t\ti\p$ and $[\ti\p_+,\ti\p]=\ti\p_+$, the co-dif\/ferential $\wt\del^*\wt\phi$ has values in $\f\t\ti\p_+$. But, by Lemma~\ref{forJosef}, it also has values in $\f\t\La^2\bar{F}$ and $\ti\p_+\cap\La^2\bar{F}=0$.

Thus it suf\/f\/ices to consider the harmonic elements from $\La^2\p_+\t\g_0$, i.e., the ones corresponding to the projective Weyl tensor.
For that purpose we consider the simple part of $Q_0=Q\cap G_0$ which is isomorphic to $\SL(n-1)$, cf.\ the matrix realisation~\eqref{q0} in the Appendix where it corresponds to the $A$-block. Considering both $\La^2\p_+\t\g_0\cong\La^2{\rr^n}^*\t{\rr^n}^*\t\rr^n$ and $\big(\La^3 f\big)[-4]\cong\La^3{\rr^n}^*$ as representations of~$\SL(n-1)$, the map~\eqref{composition} is either trivial or an isomorphism on each $\SL(n-1)$-irreducible component.

One can check that there is only one $\SL(n-1)$-irreducible component that occurs in both spaces, and it is isomorphic to $\Lambda^2{\rr^{n-1}}^*$.
Hence it suf\/f\/ices to compute $\eqref{composition}$ on one element contained in such component. Let $X_n\in\g_{-}$ and $Z_n\in\p_{+}$ be the two dual basis vectors stabilised by $\SL(n-1)$ and consider an element
\begin{gather*}
\phi= Z_1\wedge Z_2\otimes X_n\otimes Z_n-Z_1\wedge Z_2\otimes X_1\otimes Z_1+Z_n\wedge Z_2\otimes X_n\otimes Z_1.
\end{gather*}
Indeed $\phi$ is completely trace-free, satisf\/ies the algebraic Bianchi identity and the $\SL(n-1)$-orbit of $\phi$ is isomorphic to $\Lambda^2{\rr^{n-1}}^*$. Now,
\begin{gather*}
\wt\del^*\wt\phi= -\widetilde{Z}_1\otimes\widetilde{Z}_n\wedge\widetilde{Z}_2-\widetilde{Z}_n\otimes\widetilde{Z}_1\wedge\widetilde{Z}_2,
\end{gather*}
which indeed lies in the kernel of the alternation map. Hence the statement follows.
\end{proof}

We can now determine the form of the normal conformal Cartan connection:
\begin{Proposition}\label{prop-normal} The normal conformal Cartan connection is of the form
\begin{gather*}
 \wt{\om}^{\rm nor}=\wt{\om}^{\rm ind}+\Psi^1+\Psi^2,
\end{gather*}
where $\Psi^1=-\frac{1}{2}\wt{\del}^*\wt{\ka}^{\rm ind}\in\Omega^1_{\rm hor}\big(\wt\G,\ti\p\big)$ and $\Psi^2\in\Omega^1_{\rm hor}\big(\wt\G,\ti\p_+\big)$.
Furthermore, along the reduction $\G \embed \wt\G$ we have $\Psi^1\in\Om^1_{\rm hor}\big(\G,\La^2\bar{F}\big)$, $\Psi^2\in\Om^1_{\rm hor}(\G,\f)$.
\end{Proposition}

\begin{Remark}Since $\Psi^1$ and $\Psi^2$ are horizontal, they may equivalently be regarded as bundle-valued $1$-forms on~$\wt{M}$. Denoting by $\La^2 \wt{\overline{\mc{F}}}$ the associated bundle $\G\times_Q\La^2\bar{F}$ over~$\wt{M}$ and by $\wtf\subseteq T^*\wt{M}$ the annihilator of $\wt{f}=\ker\chi\subseteq T\wt{M}$, Proposition~\ref{prop-normal} says
\begin{gather*}
 \Psi^1\in\Om^1\big(\wt{M},\La^2 \wt{\overline{\mc{F}}}\big), \qquad \Psi^2\in\Om^1\big(\wt{M},\wtf\big), \qquad \Psi^1(v)=\Psi^2(v)=0, \qquad \text{for all $v\in\Ga(\ker \chi)$}.
\end{gather*}
Below we also use the corresponding frame forms, i.e., the $\wt{P}$-equivariant functions $\phi^1\colon \wt\G\goesto (\ti\g/\ti\p)^*\otimes\ti\p$ and $\phi^2\colon \wt\G\goesto (\ti\g/\ti\p)^*\otimes\ti\p_+$ such that, for any $u\in\wt\G$, $\Psi^1=\phi^1(u)\circ\wt{\om}^{\rm ind}$ and $\Psi^2=\phi^2(u)\circ\wt{\om}^{\rm ind}$. In these terms, the proposition means that along the reduction $\G\embed\wt{\G}$ these maps restrict to $Q$-equivariant functions
\begin{gather*}
\phi^1\colon \ \G\goesto \f\t\La^2\bar{F},\qquad \phi^2\colon \ \G\goesto \f\t\f.
\end{gather*}
Further we put $\Psi=\Psi^1+\Psi^2$ and $\phi=\phi^1+\phi^2$.
\end{Remark}

\begin{proof} The Kostant Laplacian $\wt\lapl$ restricts to an invertible endomorphism of $((\ti\g/\ti\p)^*\t\ti\g_0)\cap\im\wt\partial^*$ that acts by scalar multiplication on each of the $\wt G_0$-irreducible components. Now, restricting to $\G\embed\wt\G$ and suppressing all arguments $u\in\G$, it was shown in Lemma~\ref{lem-psi} that $\wt{\partial}^*\wt{\kappa}^{\rm ind}$ is contained in one of the irreducible components, namely in
$\big(f \odot \Lambda^2 f\big) [-4]$. On this component~$\wt\lapl$ acts by multiplication by~$2$. Thus, the modif\/ication map accomplishing the f\/irst normalisation step is
\begin{gather*}
\phi^1\colon \ \G\goesto \f\t\La^2\bar{F},\qquad \phi^1:=-\frac{1}{2} \wt{\partial}^*\wt{\kappa}^{\rm ind}=-\wt\lapl^{-1}\wt{\partial}^*\wt\ka^{\rm ind}.
\end{gather*}

Now, let $\wt{\om}^1:=\wt{\om}^{\rm ind}+\phi^1\circ\wt{\om}^{\rm ind}$ be the modif\/ied Cartan connection. The corresponding curvature function $\wt\ka^1$ can be expressed in terms of $\wt\ka^{\rm ind}$, $\phi^1$ and its dif\/ferential $d\phi^1$ so that
\begin{gather}
 \wt{\ka}^{1}(X,Y) =\wt{\ka}^{\rm ind}(X,Y)+\big[X,\phi^1(Y)\big]-\big[Y,\phi^1(X)\big]\notag \\
\hphantom{\wt{\ka}^{1}(X,Y) =}{} +d\phi^1(\xi)(Y)-d\phi^1(\eta)(X)-\phi^1([X,Y])+\big[\phi^1(X),\phi^1(Y)\big],\label{eqcc1}
\end{gather}
where $X,Y\in\g$ and $\xi=\big(\wt{\om}^{\rm ind}\big)^{-1}(X)$, $\eta=\big(\wt{\om}^{\rm ind}\big)^{-1}(Y)$, cf.\ \cite[formula~(3.1)]{cap-slovak-book}. For the last term we have $\big[\phi^1(X),\phi^1(Y)\big]=0$ since $\phi^1(X)$ has values in $\La^2\bar{F}$. The f\/irst three terms are
\begin{gather*}
\wt{\ka}^{\rm ind}(X,Y)+\big[X,\phi^1(Y)\big]-\big[Y,\phi^1(X)\big]=\wt{\ka}^{\rm ind}(X,Y)+\wt{\del}\phi^1(X,Y),
\end{gather*}
which by construction vanishes upon application of the Kostant co-dif\/ferential, i.e., $\wt{\del}^*\bigl(\wt{\ka}^{\rm ind}+\wt{\del}\phi^1\bigr)=0$. The remaining terms in~\eqref{eqcc1} can be combined into a~map $\La^2\g\goesto \La^2\bar{F}$,
\begin{gather*}
 (X,Y) \mapsto d\phi^1(\xi)(Y)-d\phi^1(\eta)(X)-\phi^1([X,Y]),
\end{gather*}
which vanishes upon insertion of two elements $X,Y\in \p$. Therefore, applying Lemma~\ref{lem-hor}, we conclude that $\wt{\del}^*\wt{\ka}^1$ has values in~$\f\t\f$. Thus the second modif\/ication map is
\begin{gather*}
 \phi^2\colon \ \G\goesto\f\t\f,\qquad \phi^2:=-\wt\lapl^{-1} \wt{\del}^*\wt{\ka}^1.
\tag*{\qed}
\end{gather*}
\renewcommand{\qed}{}
\end{proof}

\subsection{Properties}
The information provided in the previous proposition allows us to determine the properties satisf\/ied by the normal conformal Cartan curvature:
\begin{Proposition} The normal conformal Cartan curvature $\wt{\ka}^{\rm nor}$ restricts to a map
\begin{gather}\label{eq-curv1}
 \wt{\ka}^{\rm nor}\colon \ \G\goesto\La^2(\ti{\g}/\ti{\p})^*\t\big(\sl(n+1)\oplus \La^2 F\big).\end{gather}
Moreover, the following integrability condition holds:
\begin{gather} \label{integrability-cond}
 i_X\wt{\ka}^{\rm nor}(u)\in \f\t \big(\La^2\bar{F}\oplus \f\big), \qquad \text{for all $X\in f$, $u\in\G$}.
\end{gather}
\end{Proposition}
\begin{proof} Let $\wt{\ka}^{\rm nor}$ be the curvature function of the normal Cartan connection $\wt\om^{\rm nor}=\wt{\om}^{\rm ind}+ \phi\circ\wt{\om}^{\rm ind}$, where $\phi=\phi^1+\phi^2$. With the same conventions as in the proof of Proposition \ref{prop-normal}, \cite[formula~(3.1)]{cap-slovak-book} yields
\begin{gather*}
 \wt{\ka}^{\rm nor}(X,Y) =\wt{\ka}^{\rm ind}(X,Y)+[X,\phi(Y)]-[Y,\phi(X)] \\
\hphantom{\wt{\ka}^{\rm nor}(X,Y) =}{} +d\phi(\xi)(Y)-d\phi(\eta)(X)-\phi([X,Y])+[\phi(X),\phi(Y)].
\end{gather*}
Clearly, $\wt{\ka}^{\rm ind}(X,Y)$ has values in $\sl(n+1)$ and vanishes upon insertion of $X\in\p$. A~term of the form $[X,\phi(Y)]$ vanishes if $Y\in\p$ and has values in $\big[\p,\La^2\bar{F}\oplus \f\big]\subseteq \La^2\bar{F}\oplus\f$ for $X\in\p$. A~term of the form $d\phi(\xi)(Y)$ has values in $\La^2\bar{F}\oplus \f$ and vanishes for $Y\in \p$. The term $\phi([X,Y])$ has values in $\La^2\bar{F}\oplus \f$ and vanishes for $X,Y\in \p$. The last term $[\phi(X),\phi(Y)]$ vanishes for all $X,Y\in\g$ since $\phi(X)$ has values in $\La^2\bar{F}\oplus\f$. Altogether, we obtain~\eqref{eq-curv1} and~\eqref{integrability-cond}.
\end{proof}

We observe here that it follows directly from \eqref{eq-curv1} that the pairing of $\widetilde{\ka}^{\rm nor}$ with the involution~$K$ vanishes,
$\langle \widetilde{\ka}^{\rm nor}, K \rangle =0$.

To derive properties of induced tractorial and underlying objects on the conformal structure we will need the following preparatory lemma.
\begin{Lemma}\label{lem-cor} Let $V$ be a $\wt{G}$-representation and $v\in V$ an element which is stabilised under $G\subseteq \wt{G}$. Let $\mb{v}\in\Ga\big(\wt{\mc{V}}\big)$ be the section of the associated tractor bundle $\wt{\G}\times_{\wt{P}}V$ corresponding to the constant function $\G\goesto V$, $u\mapsto v$, along~$\G$. Then the covariant derivative $\wt{\na}^{\rm nor} \mb{v}$ corresponds to the $Q$-equivariant function
 \begin{gather*}
\G\goesto \f \t V, \qquad u\mapsto \phi^1(u)\bdot v+\phi^2(u)\bdot v.
 \end{gather*}
\end{Lemma}
\begin{proof} The covariant derivative $\wt{\na}^{\rm nor}\mb{v}$ corresponds to the map
 \begin{gather}\label{equspder}
X\in{\g} \mapsto \big(\wt{\om}^{\rm nor}\big)^{-1}(X)\cdot v+X\bdot v.
 \end{gather}
The f\/irst term in \eqref{equspder} vanishes since it is the directional derivative of the constant function $v$. Now $\wt{\om}^{\rm nor}=\wt{\om}^{\rm ind}+\phi^1+\phi^2$, and since $X \bdot v=0$ the claim follows.
\end{proof}

We now show that the distinguished tractors $\mb{s}_E$, $\mb{s}_F$ and~$\mb{K}$ on the Fef\/ferman space are all given as BGG-splittings from their underlying objects. Moreover, several stronger properties hold:

\begin{Proposition}\label{killing} Let $\mb{s}_{E}\in\Ga\big(\wt{\mc{S}}_{\pm}\big)$, $\mb{s}_{F}\in\Ga\big(\wt{\mc{S}}_{-}\big)$ and $\mb{K}\in\Ga\big(\mc{A}\wt{M}\big)$ be the tractor spinors and the adjoint tractor, respectively, and let $\eta=\Pi_0^{\wt{\mc{S}}}(\mb{s}_E)$, $\chi=\Pi_0^{\wt{\mc{S}}}(\mb{s}_F)$ and $k=\Pi_0^{\mc{A}\wt{M}}(\mb{K})$ be the corresponding underlying objects as in Proposition~{\rm \ref{prop-objects}}.
\begin{enumerate}\itemsep=0pt
\item[$(a)$] The tractor spinor $\mb{s}_F$ is parallel, i.e., $\wt{\na}^{\rm nor}\mb{s}_F=0$. In particular, $\chi$ is a pure twistor spinor, $\mb{s}_F = L_0^{\wt{\mc{S}}_{-}}(\chi)$ and $\Hol(\mb{c})\subseteq\SL(n+1)\ltimes {\Lambda}^{2}\big(\mathbb{R}^{n+1}\big)^{*}\subseteq \Spin(n+1,n+1)$.
\item[$(b)$] The tractor spinor $\mb{s}_E$ is a BGG-splitting, i.e., $\wt{\del}^*\bigl(\wt{\na}^{\rm nor}\mb{s}_E\bigr)=0$ and $\mb{s}_E = L_0^{\wt{\mc{S}}_{\pm}}(\eta)$.
\item[$(c)$] The adjoint tractor $\mb{K}$ is a BGG-splitting and $k$ is a conformal Killing field, i.e., $\wt{\partial}^* \wt{\na}^{\rm nor}\mb{K}$ $=0$, $\mb{K} = L_0^{\mc{A}\wt{M}}(k)$ and $\wt{\na}^{\rm nor} \mb{K}= \io_k \wt{\Om}^{\rm nor}$. Moreover, we have
\begin{gather}\label{Psi_proj}
\io_k \wt{\ka}^{\rm nor} = 2 \phi_{\La^2 F}.
\end{gather}
\end{enumerate}
\end{Proposition}

\begin{proof} (a) Since $\phi^1$, $\phi^2$ have values in $\ker s_F$ we have $\phi^1\bdot s_F+\phi^2\bdot s_F=0$. Thus, according to Lemma~\ref{lem-cor}, we have $\wt{\na}^{\rm nor}\mb{s}_F=0$ and the rest is obvious.

(b)~The spinor $s_E$ is of the form $s_E = \smat{ \ast \\ \eta }$. According to Lemmas~\ref{lem-psi} and~\ref{lem-cor}, $\phi^1$ has values in $\big(f\odot \La^2 f\big)[-4]$ and $\wt{\del}^*\big(\wt{\na}^{\rm nor} \mb{s}_E\big)$ corresponds to
\begin{gather*}
\wt{\partial}^*\bigl( \phi^1 \bullet s_E\bigr)
= \begin{pmatrix}
\bigl(\phi^1 \bullet \eta\bigr)_{\wt{\Sigma}_\mp[-\frac12]} \\ 0
\end{pmatrix}.
\end{gather*}
The projection $\bigl(\phi^1 \bullet \eta\bigr)_{\wt{\Sigma}_\mp[-\frac12]}$ can be realised as the full (triple) Clif\/ford action on~$\phi^1(u)\in \big(\bigotimes^3 f\big)[-4]$, where $u\in\mc{\G}$. Now it is easy to see that this action must vanish for a~$\phi^1(u)\in \big(f \odot \Lambda^2 f\big) [-4]$: We realise $\phi^1(u)$ equivalently in $\big(S^2 f \otimes f\big)[-4]$ by symmetrisation in the f\/irst two slots, then the complete Clif\/ford action on $\eta$ vanishes because the action of the f\/irst two slots is just a~(trivial) trace multiplication.

(c) According to Lemma~\ref{lem-cor}, $\wt{\na}^{\rm nor}\mb{K}$ corresponds to $\phi^1\bdot K+\phi^2\bdot K$. Since $K/\ti{\p}=k\in f$, the previous element lies in $\ti{\p}$. In particular, $\wt{\na}^{\rm nor}\mb{K}$ has trivial projecting slot, and thus $k=\Pi_0(\mb{K})$ is a conformal Killing f\/ield. Since $\phi^2\bdot K\in\ti{\p}_+$, we have that $\wt{\del}^*\bigl(\wt{\na}^{\rm nor}\mb{K}\bigr)$ corresponds to $\wt{\del}^*\bigl(\phi^1\bdot K\bigr)$. Now $\phi^1\bdot K=-K\bdot \phi^1=2\phi^1$, since $K$ acts by multiplication with~$-2$ on~$\La^2{\bar{F}}$. But $\phi^1\in\im\wt{\del}^*\subseteq\ker\wt{\del}^*$, and the expression $\wt{\del}^*\bigl(\phi^1\bdot K\bigr)$ therefore vanishes. The equality $\wt{\na}^{\rm nor} \mb{K}= \io_k \wt{\Om}^{\rm nor}$ is just \eqref{cKf} for the conformal Killing f\/ield $k$ with its BGG-splitting~$\mb{K}$. In terms of the $Q$-equivariant functions $\phi=\phi^1+\phi^2$ and $\wt{\ka}^{\rm nor}$ along $\G\embed\wt{\G}$, this can be expressed as $\phi\bdot K=i_k\wt{\ka}^{\rm nor}$, which yields~\eqref{Psi_proj}.
\end{proof}

We now collect the essential information about the induced conformal structure $(\wt{M},\mb{c})$ which we derived:

\begin{Proposition}\label{prop-properties} Let $\big(\wt{M},\mb{c}\big)$ be the conformal spin structure induced from an oriented projective structure $(M,\mb{p})$ via the Fefferman-type construction. Then the following properties are satisfied:
 \begin{enumerate}\itemsep=0pt
 \item[$(a)$] $\big(\wt{M},\mb{c}\big)$ admits a nowhere-vanishing light-like conformal Killing field $k$ such that the corresponding tractor endomorphism $\mb{K}=L_0^{\mc{A}\wt{M}}(k)$ is an involution, i.e., $\mb{K}^2=\one_{\wt{\mc{T}}}$.
 \item[$(b)$] $\big(\wt{M},\mb{c}\big)$ admits a pure twistor spinor $\chi\in\Ga\big(\wt{\Sigma}_-\big[\frac12\big]\big)$ with $k\in\Gamma(\ker\chi)$ such that the corresponding parallel tractor spinor $\mb{s}_F=L_0^{\wt{\mc{S}}_{-}}(\chi)$ is pure.
 \item[$(c)$] $\mb{K}$ acts by minus the identity on $\ker\mb{s}_F$.
 \item[$(d)$] The following integrability condition holds:
 \begin{gather}\label{condition-weyl0}\tag{W}
 v^aw^c \wt{W}_{abcd}=0, \qquad \text{for all}\ v,w\in\Ga(\ker\chi).
 \end{gather}
\end{enumerate}
\end{Proposition}

The only thing left to show for Proposition \ref{prop-properties} is that the integrability condition \eqref{integrability-cond} is equivalent to the condition \eqref{condition-weyl0} on the Weyl tensor:
\begin{Lemma}\label{lem-weyl} Let $\big(\wt{M},\mb{c}\big)$ be a split-signature conformal spin structure endowed with tractors~$\mb{s}_E$,~$\mb{s}_F$ and $\mb{K}$ satisfying conditions~$(a)$ and~$(b)$ from Proposition~{\rm \ref{prop-properties}}. Then condition~\eqref{integrability-cond} is equivalent to~\eqref{condition-weyl0}.
\end{Lemma}

\begin{proof} The implication \eqref{integrability-cond} $\Longrightarrow$~\eqref{condition-weyl0} is obvious. It remains to prove the converse implication~\eqref{condition-weyl0} $\Longrightarrow$~\eqref{integrability-cond}.

By \eqref{condition-weyl0}, one has that $ \bigl(i_X\wt{\ka}^{\rm nor}\bigr)_{\ti{\g}_0}(u)\in \big(f\t\La^2 f\big)[-4]\subseteq \f\t\La^2\bar{F}$ for $X\in f$, $u\in\G$. Since~$\mb{s}_F$ is parallel with respect to $\wt{\na}^{\rm nor}$, we have $\wt{\ka}^{\rm nor}(u)\in \La^2(\ti{\g}/\ti{\p})^*\t \bigl(\ti{\p}\cap\ker s_F\bigr)$. The projection of $\ti{\p}\cap\ker s_F$ to ${\ti{\p}_{+}}$ is precisely $\f$, hence it follows that $\bigl(i_X\wt{\ka}^{\rm nor}\bigr)_{\ti{\p}_+}(u)\in (\ti{\g}/\ti{\p})^*\t \f$, and we obtain $i_X\wt{\ka}^{\rm nor}(u)\in (\ti{\g}/\ti{\p})^*\t\big(\La^2 \bar{F}\oplus \f\big)$.

We now prove that $i_{X_1}i_{X_2}\wt{\ka}^{\rm nor}=0$ for all $X_1,X_2\in f$. For this purpose it will be useful to work with the curvature form $\wt{\Om}^{\rm nor}$, which we can represent as in~\eqref{curvature}. By~\eqref{condition-weyl0} and the algebraic Bianchi identity, $\wt{W}_{abcd}$ vanishes upon insertion of $v,w\in\Ga(\ker\chi)$ into any two slots, and in particular $v^aw^b\wt{W}_{abcd}=0$. Thus, it remains to check that $v^aw^b\wt{Y}_{dab}=0$. As in the proof of Proposition~\ref{prop-objects}, a vector f\/ield $w\in \Ga(\ker\chi)$ corresponds to a section
$\left(\begin{smallmatrix}
 * \\ w^d \\ 0
 \end{smallmatrix}\right)\in\Ga\big(\wt{\overline{\mc{F}}}\big)$.
According to \eqref{adjact},
\begin{gather*}
v^a\wt{\Om}^{\rm nor}_{ab}\bdot\begin{pmatrix}
* \\ w^d \\ 0
\end{pmatrix}=\begin{pmatrix}
-v^a \wt{Y}_{rab}w^r \\
v^a \wt{W}_{ab\; d}^{\;\;\; c} w^d \\ 0
\end{pmatrix}
\in \Ga\big(\wt{\mathbb{E}}_b\t \wt{\mc{T}}\big).
\end{gather*}
Since $i_v\wt{\Om}^{\rm nor}$ annihilates $\wt{\overline{\mc{F}}}$, it follows that $v^aw^r \wt{Y}_{rab}=0$. Using $\wt{Y}_{rab}=-\wt{Y}_{bra}-\wt{Y}_{abr}$, we obtain also $v^aw^b \wt{Y}_{rab}=0$.
\end{proof}

\subsection{Characterisation}\label{sec-backward}
We are now going to characterise the induced conformal structures. For this purpose we will introduce the following ('intermediate') Cartan connection form:
\begin{gather} \label{omti'}
\wt{\om}' := \wt{\om}^{\rm nor} -\frac{1}{2}\io_k\wt{\ka}^{\rm nor}.
\end{gather}
 The following observation then follows immediately from Proposition~\ref{prop-normal} and formula~\eqref{Psi_proj}:
 \begin{Lemma}\label{lem-same} The pullbacks of the Cartan connection forms $\wt{\om}'\in\Om^1_{\rm hor}\big(\wt{\G},\ti\g\big)$, $\wt{\om}^{\rm nor}\in\Om^1_{\rm hor}\big(\wt{\G},\ti{\g}\big)$ and $\wt{\om}^{\rm ind}\in\Om^1_{\rm hor}\big(\wt{\G},\ti{\g}\big)$ to~$\mathcal{G}\embed\wt{\G}$ agree modulo forms with values in $\p_+\subseteq\sl(n+1)\subseteq\wt{\g}$.
 \end{Lemma}

For the rest of this section, we will start with a given split-signature conformal spin structure $\big(\wt{M},\mb{c}\big)$ satisfying all the properties of Proposition~\ref{prop-properties}. In particular, $\wt{M}$ is endowed with a~conformal Killing f\/ield $k \in \Ga\big(T\wt{M}\big)$, and we can still use formula~\eqref{omti'} to def\/ine a Cartan connection~$\wt{\om}'$. The corresponding tractor connection will be denoted by~$\wt{\na}'$ and the curvature by~$\wt{\Om}'$ or~$\wt{\ka}'$. The following proposition now shows that the so constructed Cartan connection~$\wt{\om}'$ is in fact an $\SL(n+1)$-connection.

\begin{Proposition} \label{Omti'} Let $\big(\wt{M}, \mb{c}\big)$ be a split-signature conformal $($spin$)$ structure satisfying all the properties of Proposition~{\rm \ref{prop-properties}}. Then the sections~$\mb{s}_F$ and~$\mb{K}$ are parallel with respect to the tractor connection~$\wt{\na}'$, i.e., $\wt{\na}' \mb{s}_F =0$ and $\wt{\na}' \mb{K} =0$.

In particular, $\mathrm{Hol}(\wt{\om}')\subseteq \SL(n+1)\subseteq\Spin(n+1,n+1)$ and $\wt{\om}'$ pulls back to a Cartan connection of type $(\SL(n+1), Q)$ with respect to the $Q$-reduction $\G\embed \wt{\mc{G}}$. Along that reduction, the curvature functions $\wt{\ka}'$ and~$\wt{\ka}^{\rm nor}$ are related according to $\wt{\ka'} = \bigl(\wt{\ka}^{\rm nor}\bigr)_{\sl(n+1)}$ and $\wt{\ka}'$ satisfies the following integrability condition:
\begin{gather}\label{eqcur2}
i_X\wt{\ka}'(u)\in\f\t \p_+, \qquad \mbox{for all}\ X \in f,\ u\in\G.
\end{gather}
\end{Proposition}

\begin{proof} A tractor connection induced by $\wt{\om}'$ can be written as $\wt{\na}'=\wt{\na}^{\rm nor}+\Psi$ with $\Psi=-\frac{1}{2}\io_k\wt{\Om}^1$. That $\wt{\na}' \mb{s}_F =0$ follows immediately from the fact that $\wt{\na}'-\wt{\nabla}^{\rm nor}=-\frac{1}{2} i_{k}\wt{\Om}^{\rm nor}$ has values in $\Lambda^2 \wt{\mc{F}}$. Since $k$ is a conformal Killing f\/ield we have $\wt{\nabla}^{\rm nor}\mb{K}=i_{k}\wt{\Om}^{\rm nor}$. By def\/inition
\begin{gather*}
\wt{\nabla}'\mb{K}=\wt{\nabla}^{\rm nor}\mb{K}-\frac{1}{2}i_k\wt{\Om}^{\rm nor}\bdot\mb{K},
\end{gather*}
which vanishes, since $i_k\wt{\Om}^{\rm nor}$ has values in $\Lambda^2 \wt{\mc{F}}$ and therefore $\frac{1}{2}i_k\wt{\Om}^{\rm nor}\bdot\mb{K}=i_{k}\wt{\Om}^{\rm nor}$. As in Proposition~\ref{prop-properties}, we write the decomposition of $\wt{\mc{T}}$ into maximally isotropic eigenspaces of $\mb{K}$ with eigenvalues $\pm 1$ as $\wt{\mc{E}}\oplus\wt{\mc{F}}$. Since~$\mb{K}$ is $\wt{\nabla}'$-parallel, it follows that this decomposition is preserved by~$\wt{\nabla}'$. Moreover, since $\wt{\mc{F}}$ is the kernel of the pure tractor spinor~$\mb{s}_F$ it follows that $\mathrm{Hol}(\wt{\om}')\subseteq \SL(n+1)$. In particular, $\wt{\om}'$ reduces to a Cartan connection of type $(\SL(n+1), Q)$ on a~$Q$-principal bundle $\mathcal{G}\subseteq\wt{\mathcal{G}}$.

We further compute that
\begin{gather*}
\wt{\Om}' = {\wt{\Om}}^{\rm nor} - \frac12 d^{\wt{\na}^{\rm nor}} \io_k \wt{\Om}^{\rm nor}
= \wt{\Om}^{\rm nor} - \frac12 d^{\wt{\na}^{\rm nor}} \wt{\na}^{\rm nor} \mb{K}\\
\hphantom{\wt{\Om}'}{} = \wt{\Om}^{\rm nor} - \frac12 \wt{\Om}^{\rm nor} \bullet \mb{K} \notag
= \wt{\Om}^{\rm nor} + \frac12 \mb{K} \bullet \wt{\Om}^{\rm nor}
= \bigl(\wt{\Om}^{\rm nor}\bigr)_{(\wt{\mc{E}}\t\wt{\mc{F}})_0},
\end{gather*}
where we are again using $\wt{\na}^{\rm nor} \mb{K}= \io_k \wt{\Om}^{\rm nor}$ for the conformal Killing f\/ield $k$ and that $\wt{\Om}^{\rm nor}$ has values in $\wt{\mc{E}} \otimes \wt{\mc{F}} \oplus \Lambda^2\wt{\mc{F}}$. Stated for the corresponding curvature functions, this yields $\wt{\ka'} = \bigl(\wt{\ka}^{\rm nor}\bigr)_{\sl(n+1)}$. Moreover, since $\wt{\ka}^{\rm nor}$ has values in $\ker s_F\cap\wt{\p}$, it follows from~\eqref{eq-tp+p} that $( \tilde{\p} \cap \ker s_F )_{\sl(n+1)} = \p$, thus $\wt{\ka}'$ has values in~$\p$.

We know from~\eqref{integrability-cond} that $i_X \wt{\ka}^{\rm nor}$ has values in $\La^2\bar{F}\oplus \f$ for $X\in f$. But since $\big(\La^2\bar{F}\big)_{\sl(n+1)}=0$ and $(\tilde{\p}_+ \cap \ker s_F )_{\sl(n+1)}= \p_+$, we obtain that $(i_X\wt{\ka}^{\rm nor})_{\sl(n+1)}$ has values in $\p_+$. Finally, $\big(i_{X_1}i_{X_2}\wt{\ka}^{\rm nor}\big)_{\sl(n+1)}=0$ for $X_1,X_2\in f$ follows immediately from~\eqref{integrability-cond}, and altogether we obtain~\eqref{eqcur2}.
\end{proof}

Next, before proving the main characterisation Theorem~\ref{thm-char-trac}, we will show the following proposition on factorisations of particular Cartan geometries. This proposition can be understood as an adapted variant of \cite[Theorem~2.7]{cap-correspondence}.

\begin{Proposition}\label{prop-corradapt} Let $\big(\mathcal{G}\to\wt{M},\omega\big)$ be a Cartan geometry of type $(\SL(n+1),Q)$ with curvature $\ka\colon \mathcal{G}\to\Lambda^2(\g/\q)^*\otimes\g$ and let the following conditions be satisfied:
\begin{gather*}
 i_{X_1}i_{X_2}\ka(u)\in\mathfrak{p}, \qquad \text{for all}\ X_1, X_2\in\g/\q,\ u\in\G,\\
 i_{X_1}i_{X_2}\ka(u)\in\mathfrak{p}_{+}, \qquad \text{for all}\ X_1 \in \mathfrak{p}/\mathfrak{q},\ X_2 \in\g/\q,\ u\in\G,\\
 i_{X_1}i_{X_2}\kappa(u)=0, \qquad \text{for all}\ X_1,X_2\in\mathfrak{p}/\mathfrak{q},\ u\in\G.
\end{gather*}
Then $\mathcal{G}$ is locally a $P$-bundle over $M=\mathcal{G}/P$ and~$\om$ defines a~canonical projective structure on~$M$.
\end{Proposition}

\begin{proof} The third of the above listed conditions implies that~$\mathcal{G}$ is locally a~$P$-bundle $\mathcal{G}\to M$ by~\cite{cap-correspondence}. We will restrict $\mathcal{G}$ to assume this globally. We def\/ine $M=\mathcal{G}/P$ and~$\mathcal{G}_0=\mathcal{G}/P_{+}$.

Let $\sigma\colon \mathcal{G}_0\to\mathcal{G}$ be a~$G_0$-equivariant splitting. It follows from the second of the above listed conditions that
\begin{gather*}
\mathcal{L}_{\zeta_{X}}\omega=-\ad(X)\circ\omega \mod \mathfrak{p}_{+}, \qquad \text{for all $X\in\mathfrak{p}$}.
\end{gather*}
Now def\/ine
$\theta\in\Omega^1(\mathcal{G}_0,\mathfrak{g}_{-})$, $\gamma\in\Omega^1(\mathcal{G}_0,\mathfrak{g}_0)$ and $\rho\in\Omega^1(\mathcal{G}_0,\mathfrak{p}_{+})$ via the decomposition $\sigma^*\omega=\theta\oplus\gamma\oplus\rho$. Since $\sigma$ is $G_0$-equivariant and the Lie derivative is compatible with pullbacks it follows that
\begin{gather*}
\mathcal{L}_{\bar{\zeta}_{X}}(\theta\oplus\ga)=-\ad(X)\circ(\theta\oplus\ga), \qquad \text{for all $X\in\mathfrak{g}_0$}.
\end{gather*}
In particular, $\theta$ and $\gamma$ are $G_0$-equivariant and def\/ine a~(reductive) Cartan geometry $(\mathcal{G}_0\to M,\theta\oplus\gamma)$ of type $(\mathbb{R}^n\rtimes \SL(n),\SL(n))$, i.e., an af\/f\/ine connection on~$M$. Since by assumption $\Omega$ has values in $\mathfrak{p}$, $\theta\oplus\gamma$ is torsion-free and so is the af\/f\/ine connection.

Now take another splitting $\sigma'=\sigma\cdot \exp\Upsilon$, for some $\Upsilon\colon \mathcal{G}\to\mathfrak{p}_+$. Since $\Ad(\exp\Upsilon)$ acts by the identity on $\mathfrak{g}_{-}=\mathfrak{g}/\mathfrak{p}$, one has $(r^{\exp\Upsilon})^*\omega=\omega\mod\mathfrak{p}$, and thus $\theta$ is independent of the choice of splitting. Then $\sigma'^*\omega=\theta\oplus\gamma'\oplus\rho'$ and $\theta\oplus\gamma'=\Ad(\exp\Upsilon)\circ (\theta\oplus\gamma)$, projected to $\mathfrak{g}_{-}\oplus\mathfrak{g}_0$. But since $\exp\Upsilon\in P_+$, this shows that $\gamma'$ is projectively equivalent to $\gamma$. We thus obtain a~well-def\/ined projective structure on $M$.

Since $\omega$ is $P$-torsion-free and $P$-equivariant modulo $\mathfrak{p}_+$, it can be (uniquely) modif\/ied to a~normal Cartan connection $\omega^{\rm nor}\in\Omega^1(\mathcal{G},\mathfrak{g})$ with $\omega^{\rm nor}-\omega\in\Omega^1(\mathcal{G},\mathfrak{p}_{+})$. In particular, each splitting $\sigma\colon \mathcal{G}_0\to\mathcal{G}$ is in fact a~Weyl structure of the projective structure on~$M$.
\end{proof}

\begin{Theorem}\label{thm-char-trac} A split-signature $(n,n)$ conformal spin structure $\mb c$ on a~manifold~$\wt{M}$ is $($locally$)$ induced by an $n$-dimensional projective structure via the Fefferman-type construction if and only if the following properties are satisfied:
 \begin{enumerate}\itemsep=0pt
 \item[$(a)$] $\big(\wt{M},\mb{c}\big)$ admits a nowhere-vanishing light-like conformal Killing field $k$ such that the corresponding tractor endomorphism $\mb{K}=L_0^{\mc{A}\wt{M}}(k)$ is an involution, i.e., $\mb{K}^2=\one_{\wt{\mc{T}}}$.
 \item[$(b)$] $\big(\wt{M},\mb{c}\big)$ admits a pure twistor spinor $\chi\in\Ga\big(\wt{\Sigma}_-\big[\frac12\big]\big)$ with $k\in\Gamma(\ker\chi)$ such that the corresponding parallel tractor spinor $\mb{s}_F=L_0^{\wt{\mc{S}}_{-}}(\chi)$ is pure.
 \item[$(c)$] $\mb{K}$ acts by minus the identity on $\ker\mb{s}_F$.
 \item[$(d)$] The following integrability condition holds:
 \begin{gather*}
 v^aw^c \wt{W}_{abcd}=0, \qquad \text{for all}\ v,w\in\Ga(\ker\chi).
 \end{gather*}
\end{enumerate}
\end{Theorem}
\begin{proof}Starting with a projective structure $(M,\mb{p})$, it follows from Proposition \ref{prop-properties} that the induced conformal structure $\big(\wt{M},\mb{c}\big)$ has all the stated properties. On the other hand, let $\big(\wt{M},\mb{c}\big)$ be a conformal structure with the stated properties. Then, by Proposition~\ref{Omti'}, $\wt{\om}'$ restricts to a $Q$-equivariant Cartan connection form with values in $\sl(n+1)$ on the reduction $\G\embed \wt{\mc{G}}$. The corresponding curvature $\wt{\ka}'$ takes values in~$\p$ and for $X\in f$ we have that $i_X\wt{\ka}'$ takes values in~$\p_+$. It follows from Proposition~\ref{prop-corradapt} that~$\wt{\om}'$ factorises to a projective structure~$\mb{p}$ on the leaf space~$M$.

Let us now show that the two constructions are inverse to each other. Assume f\/irst that a conformal structure $(\wt{M},\mb{c})$ is induced by a~projective structure $(M,\mb{p})$. Then according to Lemma~\ref{lem-same} $\wt{\omega}'$ and $\wt{\omega}^{\rm ind}$ agree modulo~$\mathfrak{p}_+$. This implies that the projective structure def\/ined by $\wt{\omega}'$ is equal to the original projective structure. Conversely, assume now that~$(M,\mb{p})$ is a~projective structure with associated Cartan geometry $(\mathcal{G},\omega')$ that is induced from a~conformal structure~$\big(\wt{M},\mb{c}\big)$ with associated Cartan geometry $\big(\wt{\mathcal{G}},\wt{\omega}^{\rm nor}\big)$. Since $\wt{\omega}'$ is not normal, but torsion-free, there is $\varphi\in \Omega^1_{\rm hor}(\mathcal{G},\mathfrak{p}_{+})$ such that $\omega'+\varphi$ is the normal projective Cartan connection. Since $\mathfrak{p}_+\subseteq \wt{\mathfrak{p}}$ the induced conformal structure on $\wt{M}$ agrees with the original conformal structure. Thus, the Fef\/ferman-type construction (with normalisation) and the described factorisation are (locally) inverse to each other.
\end{proof}

For a reformulation of the characterisation theorem in terms of underlying objects, see Section~\ref{alternative}.

\section{Reduced scales and explicit normalisation}\label{reducedscales}

Although we obtained the desired characterisation in Theorem \ref{thm-char-trac}, we do not yet know the explicit relationship between the induced Cartan connection form $\wt{\om}^{\rm ind}$ and the normal conformal Cartan connection form $\wt{\om}^{\rm nor}$. One of the aims of the present section is to obtain a formula for this dif\/ference, which is achieved in Theorem~\ref{thm-relation}. As a consequence, we also obtain an explicit formula for the curvature $\wt{\Om}^{\rm ind}$ in terms of the normal conformal Cartan curvature $\wt{\Om}^{\rm nor}$ in Corollary~\ref{cor-curv}.
In this more ref\/ined analysis, reduced scales will play an important role.

\subsection{Characterisation of reduced scales}
The notion of reduced Weyl structures and reduced scales is introduced in Section~\ref{Relpartra}. Here we shall f\/ind an intrinsic characterisation (i.e., using the conformal structure only) of reduced scales and discuss their properties.

As the scale bundle on the projective manifold $M$ we may consider the positive elements in the density bundle~$\mbb{E}(1)$, which is the projecting part of the dual standard tractor bundle~$\mc{T}^*$, see Section~\ref{section-projective}. Similarly, on the Fef\/ferman space $\wt{M}$ we take the positive elements in the density bundle~$\wt{\mbb{E}}(1)$, the projecting part of the conformal standard tractor bundle~$\wt{\mc{T}}$. Hence for a~projective scale $\rho \in \Ga(\mbb{E}_+(1))$ we have the tractor $L_0^{\mc{T}^*}(\rho) \in \Ga(\mc{T}^*)$; similarly, for a conformal scale $\si \in \Ga\big(\wt{\mbb{E}}_+(1)\big)$ we have the tractor $L_0^{\wt{\mc{T}}}(\si) \in \Ga\big(\wt{\T}\big)$. These will be termed \emph{scale tractors}.

On the one hand, reduced scales correspond to the sections of $\mbb{E}_+(1)\to M$ seen as a subset of all sections of $\wt{\mbb{E}}_+(1)\to\wt{M}$, see Remark~\ref{rem-Weyl}. On the other hand, sections of $\mc{T}^*\to M$ are understood as specif\/ic sections of the bundle $\wt{\mc{F}}\to\wt{M}$, which is a subbundle in $\wt{\mc{T}}\to\wt{M}$, see the generalities in Section~\ref{Relpartra} and the setup of our construction in Section~\ref{construction}. It follows that these two natural inclusions commute with the BGG-splitting operators.

\begin{Lemma}\label{lem-commute}
The full arrows in the following diagram commute:
$$
\xymatrix{
\Ga(\mc{T}^*)\ \ar@{^{(}->}[r] \ar@{.>}[d]^{\Pi_0} & \Ga\big(\wt{\mc{T}}\big) \ar@{.>}[d]_{\wt{\Pi}_0} \\
\Ga(\mbb{E}_+(1))\ \ar@/^4mm/[u]^{L_0^{\mc{T}^*}} \ar@{^{(}->}[r] & \Ga\big(\wt{\mbb{E}}_+[1]\big). \ar@/_4mm/[u]_{L_0^{\wt{\mc{T}}}} \\
}
$$
\end{Lemma}
\begin{proof}
Consider a projective density $\rho \in \Ga(\mbb{E}_+(1))$ on $M$, the corresponding tractor $L_0^{\mc{T}^*}(\rho) \in \Ga(\mc{T}^*)$, and its extension to $\wt{\mc{F}}\subseteq\wt{\mc{T}}$, which is denoted by~${s}'$. The extension of $\rho \in \Ga(\mbb{E}_+(1))$ to~$\wt{\mbb{E}}_+[1]$ obviously coincides with the projection $\wt{\Pi}_0 (s')$, and it is denoted by~$\si$. We need to show that $s' = L_0^{\wt{\T}}(\si)$, i.e., that $\wt{\pa}^* \wt{\na}^{\rm nor} s'=0$. According to Proposition~\ref{prop-normal}, $\wt{\om}^{\rm nor}=\wt{\om}^{\rm ind}+\Psi^1+\Psi^2$ with
$\Psi^1\in\Om^1\big(\wt{M},\La^2 \wt{\overline{\mc{F}}}\big)$, $\Psi^2\in\Om^1\big(\wt{M},\wtf\big)$, hence we have
\begin{gather*}
 \wt{\na}^{\rm nor}s'=\wt{\na}^{\rm ind}s'+\Psi^1\bdot s'+\Psi^2\bdot s'.
\end{gather*}
Since $\La^2 \wt{\overline{\mc{F}}}$ acts trivially on $\wt{\mc{F}}\subseteq \wt{\T}$, we have $\Psi^1\bdot s'=0$. Since $\wtf\subseteq T^*\wt{M}$, it follows that $\wt{\del}^*\big(\Psi^2\bdot s'\big)=0$. It thus follows that $\wt{\del}^*\big(\wt{\na}^{\rm nor}s'\big)=\wt{\del}^*\big(\wt{\na}^{\rm ind}s'\big)$. Let $\phi$ be the frame form of~$\na L_0^{\mc{T}^*}(\rho)$. Then, according to Lemma~\ref{forJosef}, we have that $\wt{\del}^*\wt{\phi}=0$ since $\La^2F\bdot F=0$, and in particular $\wt{\del}^*\big(\wt{\na}^{\rm nor}s'\big)=0$.
\end{proof}

We can now characterise reduced scales in terms of the corresponding scale tractors:

\begin{Proposition}\label{prop-nnscales} Let $\big(\wt{M},\mb{c}\big)$ be a conformal spin structure associated to an oriented projective structure $(M,\mb{p})$ via the Fefferman-type construction. Let $\si \in \Ga\big(\wt{\mbb{E}}_+[1]\big)$ be a conformal scale and let $s := L_0^{\wt{\T}}(\si) \in\Ga\big(\wt{\T}\big)$ be the corresponding scale tractor. Then the following statements are equivalent:
\begin{enumerate}\itemsep=0pt
\item[$(a)$] The scale $\si$ is reduced.
\item[$(b)$] The tractor $s$ is a section of $\wt{\mc{F}} \subseteq \wt{\T}$.
\item[$(c)$] The twistor spinor $\chi$ is parallel with respect to the Levi-Civita connection $\wt{D}$ of the metric corresponding to the scale $\si$.
\end{enumerate}
Furthermore, in a reduced scale, the Schouten tensor is strictly horizontal, i.e., it satisfies
\begin{gather}\label{rhocondition}
v^a\wt{\Rho}_{ab}=0, \qquad \text{for all $v^a\in \Ga(\ker\chi)$},
\end{gather}
and the scalar curvature $\wt{J}$ vanishes.
\end{Proposition}

\begin{proof} (a) $\Longrightarrow$ (b): This follows from def\/initions and Lemma~\ref{lem-commute}.

(b) $\Longrightarrow$ (c): The condition (b) means that $s \cdot \mb{s}_F =0$. According to \eqref{splitStd}, \eqref{twisplit} and \eqref{traCli}, this condition expanded in slots yields
\begin{gather*}
s \cdot \mb{s}_F = L_0^{\wt{\mc{T}}}(\si) \cdot L_0^{\wt{\mc{S}}}(\chi)
= \begin{pmatrix} -\frac{1}{2n}\wt{J}\si \\ 0 \\ \si \end{pmatrix}
 \cdot \begin{pmatrix} \frac{1}{n\sqrt{2}}\crd\chi \\ \chi \end{pmatrix}
= \begin{pmatrix} -\frac{\sqrt{2}}{2n}\wt{J} \chi \si \vspace{1mm}\\ -\frac{1}{n}\crd\chi \si \end{pmatrix}
= \begin{pmatrix} 0 \\ 0 \end{pmatrix},
\end{gather*}
where we use the Levi-Civita connection $\wt D$ corresponding to $\si$. In particular, $\crd\chi=0$ and, since~$\chi$ is a twistor spinor, the condition (c) follows.

(c) $\Longrightarrow$ (b): The condition (c) yields $\mb{s}_F = \smat{0 \\ \chi}$ according to \eqref{twisplit}. The fact that $\wt{\nabla}_a \mb{s}_F=0$ yields $\wt{\Rho}_{ac}\ga^c \chi =0$ according to \eqref{traCli}. Hence \eqref{rhocondition} holds, which in particular means $\wt{J}=0$. Summarising, we have
\begin{gather}\label{ssFslots}
s=L_0^{\wt{\mc{T}}}(\si)=
\begin{pmatrix}
0 \\ 0 \\ \si
\end{pmatrix}
\qquad\text{and}\qquad
\mb{s}_F=L_0^{\wt{\mc{S}}}(\chi)=
\begin{pmatrix}
0 \\ \chi
\end{pmatrix}.
\end{gather}
Hence $s \cdot \mb{s}_F=0$ and the condition (b) follows.

(b) and (c) $\Longrightarrow$ (a): According to the previous reasoning and \eqref{constd}, we have
\begin{gather*}
\wt{\na}_a^{\rm nor} s=
\begin{pmatrix}
0 \\ \si\wt{\Rho}_{ab} \\ 0
\end{pmatrix}.
\end{gather*}
Hence $\wt{\na}^{\rm nor} s$ is strictly horizontal, i.e., $v^a \wt{\na}^{\rm nor}_a s =0$ for every $v^a\in\Ga(\ker\chi)$. Since $\wt{\na}^{\rm nor}=\wt{\na}^{\rm ind}+\Psi$ and $\Psi$ is horizontal, the horizontality of $\wt{\na}^{\rm nor} s$ is equivalent to the horizontality of~$\wt{\na}^{\rm ind} s$. Altogether, the condition~(a) follows from Proposition~\ref{mini} and Lemma~\ref{lem-commute}.
\end{proof}

We will need some f\/iner discussion on the slots of the distinguished tractor
\begin{gather} \label{Kslots}
\mb{K} = \begin{pmatrix} \rh_a \\ \mu_{ab} \, | \, \varphi \\ k_a \end{pmatrix} \in \Ga\big({\mc{A}\wt{M}}\big)
\end{gather}
in reduced scales. From Proposition \ref{killing} we know that $\mb{K}$ is the BGG-splitting $L_0^{\mc{A}\wt{M}}(k)$, which in particular means that $\mu_{ab} = \wt{D}_{[a}k_{b]}$ and $\varphi = -\frac{1}{2n} \wt{D}^rk_r$ according to~\eqref{L0La2}. However, in the following statement we only exploit the algebraic properties of $\mb{K}$, namely, that it acts by minus and plus the identity on $\wt{\mc{F}}$ and $\wt{\mc{E}}$, respectively.

\begin{Lemma}\label{Kslotsr}
Let us fix a reduced scale. Then the expression of $\mb{K}$ as in \eqref{Kslots} satisfy $\rh_a=0$, $\ph=-1$, $\mu_a{}^r v_r = -v_a$ for every $v^a \in \Ga(\ker\chi)$ and $\mu_a{}^r\mu_{rb} = g_{ab}$. Further we have $\mu_{ab} = \langle \ga_{[a} \ga_{b]} \chi, \bar{\eta} \rangle$ for some $\bar{\eta} \in \Ga\big(\wt{\Sigma}_\mp\big[{-}\tfrac12\big]\big)$.
\end{Lemma}

\begin{proof} Firstly, we use $\mb{K} \bullet s = -s$ for any $s \in \Ga\big(\wt{\mathcal{F}}\big)$. The scale tractor $s=L_0^{\wt{\mc{T}}}(\si)$ of a reduced scale $\si$ is a section of $\wt{\mc{F}}$ and it has the form as in \eqref{ssFslots}. Thus it follows from \eqref{adjact} that $\rh_a=0$ and $\ph=-1$. Next, for every $v^a \in \Ga(\ker\chi)$, the tractor $s=\smat{0\\v_a\\0}$ is clearly a section of $\wt{\mc{F}}$, since $s \cdot \mb{s}_F=0$. Thus it follows from \eqref{adjact} that $\mu_a{}^r v_r = -v_a$.

Secondly, we use $\mb{K} \bullet s = s$ for any $s \in \Ga\big(\wt{\mathcal{E}}\big)$. Considering the tractor $s=\smat{0\\\om_a\\0}$ with arbitrary $\om^a \in \Ga\big(T\wt{M}\big)$, the tractor $\bar{s} := s + \mb{K} \bullet s$ is a section of $\wt{\mathcal{E}}$, whose middle slot is $\om_a + \mu_a{}^r \om_r$. It follows again from \eqref{adjact} that $\mu_a{}^r\mu_{rb} = g_{ab}$.

Thirdly, we use \eqref{eq-K-sE-sF} which shows how $\mb{K}$ is built from $\mb{s}_E$ and~$\mb{s}_F$. Since the top slot of~$\mb{s}_F$ vanishes, middle slots of $\mb{K}$ are given by a suitable tensor product of $\chi$ and the top slot of~$\mb{s}_E$.
\end{proof}

We will also need more properties of conformal curvature quantities in reduced scales.

\begin{Lemma}\label{moreconditions} In a reduced scale,
\begin{gather}
\widetilde{W}_{abcd} \mu^{cd} = 0 , \qquad \mbox{where} \ \mu_{ab} = \wt{D}_{[a}k_{b]} ,\label{kcondition} \\
v^c \widetilde{Y}_{abc} = 0 , \qquad \mbox{for all $v^a \in \Ga(\ker\chi)$.}\label{riccicondition}
\end{gather}
\end{Lemma}

\begin{proof} In slots, the condition $\wt{\Om}_{ab}^{\rm nor} \bullet \mb{s}_F =0$ implies that $\widetilde{W}_{abcd} \ga^c \ga^d \chi=0$. Pairing both sides of the latter equality with a spinor $\bar{\eta}$ from Lemma~\ref{Kslotsr} yields~\eqref{kcondition}.

Consider two arbitrary sections $v^a,w^b$ of $\wt{f}= \ker\chi$. Conditions \eqref{condition-weyl0} and \eqref{rhocondition} imply $\wt{R}_{abcd} v^a w^d = 0$. Now, inserting $v^a$ and $w^e$ into the Bianchi identity $\wt{D}_{[a} \wt{R}_{bc]de} = 0$, we obtain $v^a w^e \wt{D}_a \wt{R}_{bcde} =0 $,
where we used the fact that $\wt{f}$ is parallel. Since we can always regard~$g^{ab}$ as a~section of~$\wt{f} \otimes \wt{f}^*$, this implies
$0= g^{ec} v^a \wt{D}_a \wt{R}_{bcde} = v^a \wt{D}_a \wt{\Ric}_{bc}$ where $\wt{\Ric}_{bc}$ is the Ricci tensor of~$\wt{D}_a$. Since $\wt{J} =0$, we have that $\wt{\Rho}_{ab}$ is proportional to~$\wt{\Ric}_{ab}$ by a constant factor. Thus $v^a \wt{D}_a \wt{\Rho}_{bc} = 0$.
From~\eqref{rhocondition} and since $\wt{f}$ is parallel we also have $v^b \wt{D}_a \wt{\Rho}_{bc} =0$. Altogether, \eqref{riccicondition}~follows by the def\/inition of the Cotton tensor.
\end{proof}

\subsection{Explicit normalisation formula}\label{explicit}
So far we discussed three Cartan connections on the Fef\/ferman space $\wt{M}$: the induced one $\wt{\om}^{\rm ind}$ (Section~\ref{sec-feff}),
the corresponding normal one~$\wt{\om}^{\rm nor}$ (Section~\ref{nprocess}) and the modif\/ied auxiliary one $\wt{\om}'$ (Section~\ref{sec-backward}). Various properties of these and derived objects are enumerated in Propositions~\ref{prop-properties} and~\ref{Omti'}. The following proposition ref\/ines the integrability conditions included there.

\begin{Proposition} Let $\big(\wt{M},\mb{c}\big)$ be the conformal spin structure induced from an oriented projective structure $(M,\mb{p})$ via the Fefferman-type construction. Then, along the reduction $\G\embed \wt{\G}$,
\begin{gather} \label{integrability-refin2}
 \io_X\wt{\ka}^{\rm nor}(u) \in \f\t\La^2\bar{F}, \qquad \mbox{for all}\ X \in f,\ u\in\G,\\
\label{integrability-refin3} \io_X\wt{\ka}'(u)=0, \qquad \mbox{for all}\ X \in f,\ u\in\G.
\end{gather}
\end{Proposition}

\begin{proof} From \eqref{integrability-cond} we already know that $\io_X\wt{\ka}^{\rm nor}$ has values in $ \f\t \big(\La^2\bar{F}\oplus \f\big)$.
We note that the top slot of sections of $\La^2\wt{\overline{\mc{F}}}$ vanishes in reduced scales, cf.~\eqref{useful}. Thus the part in~$\f$ corresponds to~$v^r\wt{Y}_{abr}$ for a~$v\in\Ga(\wt{f})$, which however has to vanish by~\eqref{riccicondition}. Hence \eqref{integrability-refin2} follows. The last condition~\eqref{integrability-refin3} follows from $\wt{\ka}' = \big(\wt{\ka}^{\rm nor}\big)_{\sl(n+1)}$, cf.\ Proposition~\ref{Omti'}.
\end{proof}

Since $\wt{\om}'$ is an $\SL(n+1)$-connection on $\wt{\G}\to\wt{M}$, it is the extension of a Cartan connection~$\om'$, on~$\G\to\wt{M}$. Now, due to \eqref{integrability-refin3}, any section $v\in\Ga(\ker\chi)$ inserts trivially into its curvature. But this is the standard condition on the connection~$\om'$ to be a Cartan connection also on the bundle~$\G\to M$, i.e., to be a projective Cartan connection, cf.~\cite{cap-correspondence}.

Furthermore, we will show that the descended Cartan connection is normal, i.e., $\om'=\om$. To do this, we f\/irst compute $\wt{\pa}^* \wt{\ka}'$ and then use the relation between the co-dif\/ferentials $\pa^*$ on $M$ and $\wt{\pa}^*$ on $\wt{M}$ discussed in Lemma~\ref{forJosef}.

\begin{Proposition} \label{pa*Om'}
The curvature $\wt{\ka}'$ satisfies
\begin{gather}\label{integrability-refine}
 \wt{\pa}^* \wt{\ka}' (u) = \io_k \wt{\ka}^{\rm nor}(u) \in \f \t \La^2\bar{F}, \qquad \mbox{for\ all}\ u\in\G.
\end{gather}
\end{Proposition}

\begin{proof} We shall compute $\wt{\pa}^* \wt{\Om}'$ directly. First observe that using Proposition~\ref{Omti'} we have $\wt{\Om}' = \big(\wt{\Om}^{\rm nor}\big)_{\sl(n+1)} = \wt{\Om}^{\rm nor} + \frac12 \mb{K} \bullet \wt{\Om}^{\rm nor}$. Hence $\wt{\pa}^*\wt{\Om}' = \frac12 \wt{\pa}^* \big( \mb{K} \bullet \wt{\Om}^{\rm nor} \big)$, since $\wt{\pa}^* \wt{\Om}^{\rm nor}=0$. Using~\eqref{Kslots} and~\eqref{curvature}, we compute $\mb{K} \bullet \wt{\Om}^{\rm nor}_{ab}$ as
\begin{gather*}
\begin{pmatrix} \rh_c \\ \mu_{c_0c_1} \, | \, \varphi \\ k_c \end{pmatrix} \bullet
\begin{pmatrix} -\wt{Y}_{dab} \\ \wt{W}_{abd_0d_1} \,|\, 0 \\ 0 \end{pmatrix}
= \begin{pmatrix}
\rh^r\wt{W}_{abrc} - \mu_c{}^r \wt{Y}_{rab} + \ph \wt{Y}_{cab} \\
-2\wt{W}_{ab}{}^r{}_{[c_0}\mu_{c_1]r} + 2 k_{[c_0}\wt{Y}_{c_1]ab} \mid k^r\wt{Y}_{rab} \\
k^r\wt{W}_{abrc} \end{pmatrix}.
\end{gather*}
In a reduced scale, from the previous display together with Lemmas~\ref{Kslotsr} and~\ref{moreconditions} we compute
\begin{gather*}
\wt{\pa}^* \big(\mb{K} \bullet \wt{\Om}^{\rm nor}_{ab}\big)
= \begin{pmatrix} 0 \\ 2k^r\wt{W}_{rac_0c_1} \mid 0\\ 0 \end{pmatrix}
= 2 k^r \wt{\Om}^{\rm nor}_{ra},
\end{gather*}
which yields \eqref{integrability-refine}.
\end{proof}

\begin{Theorem}\label{thm-relation} Let $(\G,\om)$ be a projective normal Cartan geometry over $M$ and let $\big(\wt{\G},\wt{\om}^{\rm ind}\big)$ be the conformal Cartan geometry over $\wt{M}$ induced via the Fefferman-type construction. Then
\begin{enumerate}\itemsep=0pt
\item[$(a)$] $\wt{\om}^{\rm ind}=\wt{\om}'=\wt{\om}^{\rm nor}-\frac{1}{2}i_k\wt{\ka}^{\rm nor}$,
\item[$(b)$] $\wt{\om}^{\rm nor}=\wt{\om}^{\rm ind}+\Psi^1$, where $\Psi^1=-\frac{1}{2}\wt{\del}^*\wt{\ka}^{\rm ind}=\frac{1}{2}\io_k\wt{\ka}^{\rm nor}$.
\end{enumerate}
\end{Theorem}
\begin{proof} (a) We use that $i_X\wt\ka'=0$ for all $X\in f$ according to~\eqref{integrability-refin3}. Then Proposition~\ref{pa*Om'} together with Lemma~\ref{forJosef} imply that $\delstar\ka'=0$. Thus $\om'$ is projectively normal, and therefore we obtain $\wt{\om}'=\wt{\om}^{\rm ind}$.

(b) The normalisation process of Proposition \ref{prop-normal} provides $\Psi=\Psi^1+\Psi^2$ such that $\wt{\om}^{\rm nor} = \wt{\om}^{\rm ind} + \Psi$, where $\Psi^1$, $\Psi^2$ are the f\/irst and second normalisation steps. However since $\wt{\om}' = \wt{\om}^{\rm ind}$, it follows from Proposition~\ref{pa*Om'} and~\eqref{omti'} that $\wt{\pa}^* \wt{\ka}'=\wt{\pa}^* \wt{\ka}^{\rm ind}$ is, up to a constant multiple, the dif\/ference between $\wt{\om}^{\rm nor}$ and $\wt{\om}^{\rm ind}$. Therefore already the f\/irst normalisation step completes the normalisation, i.e., $\Psi^2=0$.
\end{proof}

Using the explicit relationship provided in Theorem \ref{thm-relation} we can also obtain a detailed description of the dif\/ference between the induced and the normal Cartan curvatures:
\begin{Corollary}\label{cor-curv} In a reduced scale, we have the following relation between the curvatures of the induced and the normal conformal Cartan connection:
\begin{gather}\label{IndCartanCurv}
\wt{\Om}^{\rm ind}_{ab} = \wt{\Om}^{\rm nor}_{ab} + \frac12 \mb{K} \bullet \wt{\Om}^{\rm nor}_{ab}=
\begin{pmatrix}
-\wt{Y}_{cab} \\ \wt{W}_{abc_0c_1} - \wt{W}_{ab}{}^r{}_{[c_0}\mu_{c_1]r} + k_{[c_0} \wt{Y}_{c_1]ab} \mid 0 \\
\frac12 k^r \wt{W}_{abrc}
\end{pmatrix}.
\end{gather}
In particular, $\frac12 \io_k \wt{W}$ is the torsion of the induced Cartan connection $\wt{\om}^{\rm ind}$.
\end{Corollary}

\begin{proof}We obtained the concrete expression of $\mb{K} \bullet \wt{\Om}^{\rm nor}$ in the proof of Proposition~\ref{pa*Om'}. Now Lemmas~\ref{Kslotsr} and~\ref{moreconditions}, and a~short computation yields~\eqref{IndCartanCurv}.
\end{proof}

\section[Comparison with Patterson--Walker metrics and alternative characterisation]{Comparison with Patterson--Walker metrics\\ and alternative characterisation}\label{Comparing}

In this section we will show that the Fef\/ferman-type construction studied in this article is closely related to the construction of so-called \emph{Patterson--Walker metrics}. These are the Riemann extensions of af\/f\/ine connected spaces, f\/irstly described in~\cite{patterson-walker}. A~conformal version of this construction was obtained by~\cite{dunajski-tod} for dimension $n=2$, and was treated by the authors of the present article in general dimension in~\cite{hsstz-walker}.

\subsection{Comparison}\label{comparison}
Let $M$ be a smooth manifold and $p\colon T^*M\to M$ its cotangent bundle. The vertical subbundle $V\subseteq T(T^*M)$ of this projection is canonically isomorphic to $T^*M$. An af\/f\/ine connection~$D$ on~$M$ determines a complementary horizontal distribution $H\subseteq T(T^*M)$ that is isomorphic to~$TM$ via the tangent map of~$p$.

\begin{Definition}\label{walker} The \emph{Riemann extension} or the \emph{Patterson--Walker metric} associated to a~tor\-sion-free af\/f\/ine connection~$D$ on~$M$ is the pseudo-Riemannian metric~$g$ on~$T^* M$ fully determined by the following conditions:
 \begin{enumerate}\itemsep=0pt
 \item[(a)] both $V$ and $H$ are isotropic with respect to $g$,
 \item[(b)] the value of $g$ with one entry from $V$ and another entry from $H$ is given by the natural pairing between $V\cong T^*M$ and $H\cong TM$.
 \end{enumerate}
\end{Definition}
It follows that $V$ is parallel with respect to the Levi-Civita connection of the just constructed metric. Hence Patterson--Walker metrics are special cases of Walker metrics, i.e., metrics admitting a parallel isotropic distribution. The explicit description of the metric~$g$ in terms of the Christof\/fel symbols of~$D$ can be found in~\cite{hsstz-walker, patterson-walker}.

The previous def\/inition can be adapted to weighted cotangent bundles $T^*M(w)=T^*M\otimes\mbb{E}(w)$, provided that $M$ is oriented and $D$ is special, i.e., preserving a volume form on $M$, which allows a trivialisation of $\mbb E(w)$. It turns out that Patterson--Walker metrics induced by projectively equivalent connections are conformally equivalent if and only if $w=2$ (interpreted as a projective weight according to the conventions from Section~\ref{section-projective}). Altogether, we have a natural split-signature conformal structure on $T^*M(2)$ induced by an oriented projective structure $(M,\mb{p})$.

From Section \ref{sec-feff} we know that $\wt{M}=T^*M(2)\setminus\{0\}$ is the Fef\/ferman space. Special af\/f\/ine connections from $\mb{p}$ are just the exact Weyl connections of the corresponding parabolic geometry. The corresponding objects on $\wt M$ are the reduced Weyl connections, respectively reduced scales, which correspond to distinguished metrics in the conformal class, see Section~\ref{Relpartra}. We are going to show that these metrics are just the Patterson--Walker metrics.

\begin{Proposition} \label{prop-feff-walker}
Let $\big(\wt M,\mb c\big)$ be the conformal structure of signature $(n,n)$ associated to an $n$-dimensional projective structure $(M,\mb p)$ via the
Fefferman-type construction. Then any metric in $\mb c$ corresponding to a reduced scale is a~Patterson--Walker metric.
\end{Proposition}
\begin{proof} Within the proof we refer to the notation and explicit matrix realisations from Appendix~\ref{appendixB}. By def\/inition, the Fef\/ferman space is $\wt{M}=\G/Q$, which yields $T\wt{M}\cong\G\x_Q\g/\q$. Conformally invariant objects on~$\wt{M}$, respectively objects related to a choice of reduced scale, correspond to data on $\g/\q\cong\ti\g/\ti\p$ that are invariant under the action of~$Q$, respectively $G_0^{ss}\cap Q$. Elements in $\g/\q$ will be represented by matrices of the form
\begin{gather*}
\pmat{-\frac{z}2&*&*\\X&*&*\\w&Y^t&-\frac{z}2},
\end{gather*}
where $z,w\in\rr$ and $X,Y\in\rr^{n-1}$. Firstly, one verif\/ies that
\begin{gather} \label{kl-mn}
Y^tX-zw,
\end{gather}
is the only quadratic form that is invariant under $G_0^{ss}\cap Q$. Hence any reduced-scale metric in~$\mb{c}$ corresponds to the quadratic form~\eqref{kl-mn} in a suitable frame. Secondly, the vertical subbundle $V\subseteq T\wt{M}$ corresponds to the $Q$-invariant subspace $f=\p/\q\subseteq\g/\q$ given by $X=0$ and $w=0$. The horizontal distribution $H\subseteq T\wt{M}$ induced by a linear connection from $\mb{p}$ corresponds to the unique $(G_0^{ss}\cap Q)$-invariant subspace $h\subseteq\g/\q$ complementary to~$f$, which is given by $Y=0$ and $z=0$. Obviously, both $f$ and $h$ are isotropic with respect to~\eqref{kl-mn}. Hence any reduced-scale metric in $\mb{c}$ satisf\/ies the condition (a) from Def\/inition~\ref{walker}. Thirdly, the canonical identif\/ication $V\cong T^*M(2)$ corresponds to an isomorphism $f\cong(\g/\p)^*(2)$ of $Q$-modules. Identifying $(\g/\p)^*(2)$ with $\p_+(2)$, it turns out to be given by
\begin{gather*}
\pmat{-\frac{z}2&*&*\\0&*&*\\0&Y^t&-\frac{z}2}
\mapsto \pmat{0&Y^t&-z\\0&0&0\\0&0&0}.
\end{gather*}
Now, the inner product of any $v\in f$ and $u\in h$ coincides with the pairing of the corresponding elements $v\in\p_+(2)$ and $u\in\g/\p$. Hence any reduced-scale metric in $\mb{c}$ satisf\/ies also the condition~(b) from Def\/inition~\ref{walker} and so it is a Patterson--Walker metric.
\end{proof}

\subsection{Alternative characterisation}\label{alternative}
We have characterised split-signature $(n,n)$ conformal structures $\mb c$ on $\wt{M}$ induced by an $n$-dimensional projective structure via the Fef\/ferman-type construction in Theorem \ref{thm-char-trac}. Now we know these structures correspond to conformal Patterson--Walker metrics discussed in~\cite{hsstz-walker}. There we found the following characterisation in terms of underlying objects by direct computations and spin calculus. Our aim here is to indicate how to reach the same result in the current framework.

\begin{Theorem}\label{char-refin} A split-signature $(n,n)$ conformal spin structure $\mb c$ on a manifold $\wt{M}$ is $($locally$)$ induced by an $n$-dimensional projective structure via the Fefferman-type construction if and only if the following properties are satisfied:
\begin{enumerate}\itemsep=0pt
\item[$(a)$] $\big(\wt{M},\mb{c}\big)$ admits a nowhere-vanishing light-like conformal Killing field $k$.
\item[$(b)$] $\big(\wt{M},\mb{c}\big)$ admits a pure twistor spinor $\chi$ such that $\wt{f}=\ker\chi$ is integrable and
$k\in\Ga\big(\wt{f}\big)$.
\item[$(c)$] The Lie derivative of $\chi$ with respect to the conformal Killing field $k$ is
$\mathcal{L}_k \chi = -\frac12 (n+1) \chi$.
\item[$(d)$] The following integrability condition holds:
\begin{gather}\tag{\ref{condition-weyl0}}\label{condition-weyl}
v^rw^s \wt{W}_{arbs}=0, \qquad \mbox{for all}\ v^r, w^s\in\Gamma(\ker\chi).
\end{gather}
\end{enumerate}
\end{Theorem}

We now express the conditions from Proposition \ref{prop-properties} in underlying terms:

(i) For a conformal Killing f\/ield $k$ with the splitting $\mathbf{K}$ as in \eqref{Kslots}, a~straightforward computation shows that the condition $\mathbf{K}^2 = \id_{\wt{\mc{T}}}$ is equivalent to
\begin{alignat}{3}
& k^a k_a = 0 , \qquad && \rho^a \rho_a = 0 ,& \nonumber \\
& \mu \ind{^a_b} k^b = \varphi k^a , \qquad && \mu \ind{^a_b} \rho^b = - \varphi \rho^a , & \nonumber\\
& k^a \rho_a = \varphi^2 - 1 , \qquad && \mu \ind{_a^c} \mu_{cb} = g_{ab} + 2 k_{(a} \rho_{b)} .& \label{Ksquared}
\end{alignat}

(ii) For a twistor spinor $\chi$, the corresponding tractor spinor $\mb{s}_F= \left(\begin{smallmatrix} \bar\chi \\ \chi \end{smallmatrix}\right)\in\Ga(\wt{\mc{S}}_{-})$ is parallel with respect to $\wt{\na}^{\rm nor}$. In particular, purity of $\mb{s}_F$ can be checked at one point. If $\chi=0$, respectively $\bar\chi=0$, this tractor spinor is pure whenever $\bar\chi$, respectively $\chi$, is pure. If $\chi\not=0$ and $\bar\chi\not=0$, the purity of $\left(\begin{smallmatrix} \bar\chi \\ \chi \end{smallmatrix}\right)$ is equivalent to $\chi$ and $\bar\chi$ being pure and their kernels having $(n-1)$-dimensional intersection, cf.\ \cite[Proposition~III-1.12]{chevalley-spinors} or~\cite{hughston-mason,Taghavi-Chabert2015}.

(iii) Let $k$ be a conformal Killing f\/ield which splits to $\mathbf{K}$ and $\chi$ a twistor spinor which splits to $\mb{s}_F$. Then the condition
$\mc{L}_k\chi=-\frac{1}{2}(n+1)\chi$ is equivalent to $\mb{K}\bdot \mb{s}_F=-\frac{1}{2}(n+1)\mb{s}_F$. If the tractor spinor $\mb{s}_F$ is pure it has an $(n+1)$-dimensional maximally isotropic kernel $\ker \mb{s}_F$. Then $\mb{K}\bdot \mb{s}_F=-\frac{1}{2}(n+1)\mb{s}_F$ is equivalent to $\mb{K}$ acting by minus the identity on $\ker \mb{s}_F$, which therefore coincides with the eigenspace of $\mb{K}$ corresponding to~$-1$.

The assumption on the pure twistor spinor $\chi$ in Theorem~\ref{char-refin} guarantees the existence of a suitable compatible metric for which $\chi$ is parallel. This result is proved in \cite[Proposition~4.2]{hsstz-walker}. Henceforth we shall assume $\wt{D}\chi=0$ where $\wt{D}$ is the corresponding Levi-Civita connection. In particular, we have $\mb{s}_F=L_0^{\wt{\mc{S}}_-}(\chi) = \smat{0 \\ \chi}$, which is pure and parallel by observation (ii). Expanding the latter condition according to \eqref{traCli} yields
\begin{gather}\label{rhocondition2}
v^a\wt{\Rho}_{ab}=0, \qquad \text{for all $v^a\in \Ga(\ker\chi)$},
\end{gather}
For $\mb{K} = L_0^{\mc{A}\wt{M}}(k)$, we know from observation (iii) that $\mb{K}$ acts by $-\id$ on $\ker \mb{s}_F$. By the very same reasoning as in the f\/irst part of the proof of Lemma~\ref{Kslotsr} it follows that
\begin{gather}\label{muv}
\rh_a=0,\qquad \ph=-1,\qquad \mu_a{}^r v_r = -v_a, \qquad \text{for all $v^a \in \Ga(\ker\chi)$}.
\end{gather}
Now we are prepared to prove the theorem:

\begin{proof}[Proof of Theorem \ref{char-refin}] If $(\wt{M},\mb{c})$ is induced by a projective structure, the stated properties hold according to Proposition~\ref{prop-properties} and previous observations. For the converse direction, it remains to show that $\mb{K}^2 = \id_{\wt{\mc{T}}}$, which is equivalently characterised by the identities~\eqref{Ksquared}. Then the properties (a)--(d) of Proposition~\ref{prop-properties} will be satisf\/ied and the result will follow from the characterisation Theorem~\ref{thm-char-trac}.

The expansion of the prolonged conformal Killing equation \eqref{cKf} for $k$ gives
\begin{gather}
\wt{D}_a k_b = \mu_{ab} + g_{ab}, \label{prolong1} \\
\wt{D}_a\mu_{br} = -2 \wt{\Rho}_{a[b} k_{r]} - W_{bras}k^s, \label{prolong2}
\end{gather}
according to \eqref{conadj} and \eqref{curvature}. Next, from \eqref{muv} we especially have $\mu_b{}^r k_r = -k_b$. Applying $\wt{D}_a$ to both sides of this equality and using \eqref{prolong1} we obtain
\begin{gather*}
\big(\wt{D}_a\mu_{br}\big) k^r + \mu_b{}^r \mu_{ar} + \mu_{ba} = -(\mu_{ab} + g_{ab}).
\end{gather*}
From \eqref{prolong2}, \eqref{condition-weyl} and \eqref{rhocondition2} we have $\big(\wt{D}_a\mu_{br}\big) k^r =0$, hence the previous display shows $\mu_a{}^r \mu_{rb} = g_{ab}$. This together with \eqref{muv} implies that all identities from~\eqref{Ksquared} are satisf\/ied, hence $\mb{K}^2 = \id_{\wt{\mc{T}}}$.
\end{proof}

\appendix

\section{Explicit matrix realisations}\label{appendixB}
Here we provide explicit realisations of the Lie algebras introduced in Section~\ref{construction} in terms of matrices. We will consider the inner product $h$ and the involution $K$ on $\rr^{n+1,n+1}$ given by
the block matrices
\begin{gather*}
 h:= \begin{pmatrix}
 0 & I_{n+1} \\
 I_{n+1} & 0
\end{pmatrix}
\qquad\text{and}\qquad
K:= \begin{pmatrix}
 I_{n+1} & 0 \\
 0 & -I_{n+1}
\end{pmatrix}
\end{gather*}
with respect to the standard basis $(e_1,\dots,e_{2n+2})$. Then $E=\langle e_1,\dots,e_{n+1}\rangle$ and $F=\langle e_{n+2},\dots$, $e_{2n+2}\rangle$ and the decomposition~\eqref{tig-decomp} can be written as
\begin{gather*}
\ti\g=\Lambda^2 (E\oplus F)= \begin{pmatrix}
 E\otimes F& \Lambda^{2}E\\
 \Lambda^{2}F &E\otimes F
 \end{pmatrix}.
\end{gather*}
For $\ti v:=e_1+e_{2n+2}$, the Lie algebra $\ti\p$ of the parabolic subgroup $\wt{P}\subseteq\wt{G}$ is of the following form
\begin{gather}\label{parabolic}
\ti\p=
\left(\begin{array}{@{}ccccccc@{}}
 a & U^t & w & \vline & 0 & -W^t & -b \\
 X & B & V & \vline & W & C & -X \\
 0 & Y^t & c & \vline & b & X^t & 0 \\
 \hline
 0 & -Y^t & -d & \vline & -a & -X^t & 0 \\
 Y & D & -Z & \vline & -U & -B^t & -Y \\
 d & Z^t & 0 & \vline & -w & -V^t & -c
\end{array}\right),
\end{gather}
where $a,b,c,d,w\in\rr$ with $a-b=d-c$, $U, V, W, X, Y, Z\in\rr^{n-1}$, $B\in\gl(n-1)$ and $C,D\in\so(n-1)$.
The nilradical $\ti\p_+=\ti\p^{\perp}$ is then of the form
\begin{gather*}
\ti\p_+=
\left(\begin{array}{@{}ccccccc@{}}
 a & U^t & w & \vline & 0 & -V^t & -a \\
 0 & 0 & V & \vline & V & 0 & 0 \\
 0 & 0 & a & \vline & a & 0 & 0 \\
 \hline
 0 & 0 & -a & \vline & -a& 0 & 0 \\
 0 & 0 & -U & \vline & -U & 0& 0 \\
 a & U^t & 0 & \vline & -w & -V^t & -a
 \end{array}\right).
\end{gather*}
A choice of Levi subalgebra $\ti\g_0\subseteq\ti\p$ determines a grading $\ti\g=\ti\g_-\oplus\ti\g_0\oplus\ti\p_+$.
We shall choose $\ti\g_0=\ti\p\cap\ti\p_{\rm op}$, where $\ti\p_{\rm op}\subseteq\ti\g$ is the stabiliser of the light-like vector $e_{n+2}$.
Explicitly,
\begin{gather*}
\ti\g_0=
\left(\begin{array}{@{}ccccccc@{}}
 a & 0 & 0 & \vline & 0 & 0 & 0\\
 X & B & V & \vline & 0 & C & -X\\
 0 & Y^t & c & \vline & 0 & X^t & 0\\
 \hline
 0 & -Y^t & -a-c & \vline & -a & -X^t & 0\\
 Y & D & -Z & \vline & 0 & -B^t & -Y\\
 a+c & Z^t & 0 & \vline & 0 & -V^t & -c
 \end{array}\right).
\end{gather*}

The embedding $i'\colon \g\embed\ti{\g}$ of Lie algebras has the form
$A \mapsto \left(\begin{smallmatrix} A& 0 \\ 0 & -A^t \end{smallmatrix}\right)$.
The subgroup $Q=i^{-1}(\wt{P})$ is contained in $P$, the stabiliser in $G$ of $v=(\ti{v})_E=e_1$; the inclusion of corresponding Lie algebras is
\begin{gather*}
 \q=\g\cap\ti{\p}=\pmat{a&U^t&w\\0&A&V\\0&0&-a}\subseteq\pmat{a&U^t&w\\0&B&V\\0&X^t&c}=\p,
\end{gather*}
where $\tr(A)=0$ and $a+\tr(B)+c=0$. The standard projective grading $\g=\g_-\oplus\g_0\oplus\p_+$,
\begin{gather*}
\g_{-} = \pmat{0&0&0\\X&0&0\\y&0&0}, \qquad
\g_0= \pmat{a&0&0\\0&B&V\\0&X^t&c}, \qquad
\p_+= \pmat{0&U^t&w\\0&0&0\\0&0&0}, \qquad
\end{gather*}
is compatible with the previous conformal grading so that the reduced Lie subalgebra $\q_0:=\q\cap\g_0$ coincides with the intersection of $\g_0\cap\ti\g_0$. Explicitly,
\begin{gather}\label{q0}
\q_0=\pmat{a&0&0\\0&A&V\\0&0&-a},
\end{gather}
where $\tr(A)=0$.

\subsubsection*{Acknowledgements}
The authors express special thanks to Maciej Dunajski for motivating the study of this construction and for a number of enlightening discussions on this and adjacent topics. KS would also like to thank Pawe{\l}~Nurowski for drawing her interest to the subject and for many useful conversations. MH gratefully acknow\-led\-ges support by project P23244-N13 of the Austrian Science Fund (FWF) and by `Forschungsnetzwerk Ost' of the University of Greifswald. KS gratefully acknowledges support from grant J3071-N13 of the Austrian Science Fund (FWF). J\v{S}~was supported by the Czech science foundation (GA\v{C}R) under grant P201/12/G028. AT-C was funded by GA\v{C}R post-doctoral grant GP14-27885P. V\v{Z}~was supported by GA\v{C}R grant GA201/08/0397. Finally, the authors would like to thank the anonymous referees for their helpful comments and recommendations.

\LastPageEnding

\end{document}